\documentclass[12pt]{article}
\usepackage[T1]{fontenc}
\usepackage[utf8]{inputenc}

\DeclareMathSizes{12}{30}{16}{12}

\usepackage{amsfonts}
\usepackage{amsmath}
\usepackage{framed}

\usepackage[english]{babel}
\usepackage{amsthm}
\usepackage{framed}
\usepackage{tikz}

\usepackage{amsmath}
\usepackage{amsfonts}
\usepackage{amssymb}
\usepackage{mathtools}
\usepackage{commath}
\usepackage{graphicx} 
\usepackage[sc,osf]{mathpazo}

\newcommand{\R}{\mathbb{R}}

\newtheorem{theorem}{Theorem}
\newtheorem{lemma}{Lemma}
\newtheorem{proposition}{Proposition}
\newtheorem{corollary}{Corollary}

\newtheorem{prop}{Prop}

\newcommand{\barc}{\bar c_\eta} 
\newcommand{\uc}{\underline{c}_\eta} 
\newcommand{\bulk}{{\mathrm {bulk}}} 

\newcommand{\poubelle}[1]{}

\newcommand{\cI}{c_I}
\newcommand{\cA}{c_A}
\newcommand{\cgamma}{c_\gamma}
\newcommand{\Cgamma}{C_\gamma}

\newcommand{\cL}[1]{c_{#1}}
\newcommand{\indep}{\perp \!\!\! \perp}

\newcommand*\circled[1]{\tikz[baseline=(char.base)]{
            \node[shape=circle,draw,inner sep=2pt] (char) {#1};}}

\let\norm\undefined 
\DeclarePairedDelimiter\norm{\lVert}{\rVert}

\newcommand{\revision}[1]{\textcolor{black}{#1}}

\usepackage{authblk}
\usepackage{lmodern}

\title{Sharp Local Minimax Rates for Goodness-of-Fit Testing in multivariate Binomial and Poisson families and in multinomials}
\author[1]{Julien Chhor }
\author[2]{Alexandra Carpentier}
\affil[1]{CREST/ENSAE}
\affil[2]{OvGU, Magdeburg}
\date{}

\begin{document}

\maketitle

Contact: julien.chhor@ensae.fr, carpentier@uni-potsdam.de

\newpage

\begin{abstract}
We consider the identity testing problem - or goodness-of-fit testing problem - in multivariate binomial families, multivariate Poisson families and multinomial distributions. 
Given a known distribution $p$ and $n$ iid samples drawn from an unknown distribution $q$, we investigate how large $\rho>0$ should be to distinguish, with high probability, the case $p=q$ from the case $d(p,q) \geq \rho $, where $d$ denotes a specific distance over probability distributions. 
We answer this question in the case of a family of different distances: $d(p,q) = \norm{p-q}_t$ for $t \in [1,2]$ where $\|\cdot\|_t$ is the entrywise $\ell_t$ norm. Besides being locally minimax-optimal - i.e. characterizing the detection threshold in dependence of the known matrix $p$ - our tests have simple expressions and are easily implementable. 
\end{abstract}


\section{Introduction}

\revision{We consider the problem of \textit{identity testing} or \textit{goodness-of-fit testing} in multivariate binomial families, multivariate Poisson families and multinomial distributions.
At a high level, this problem aims at testing whether or not the data distribution matches a given known distribution. 
Throughout the paper, we will state the results in the multivariate binomial setting, and will establish the link with multivariate Poisson families and multinomials later on. 
The problem can be stated as follows: given $n$ i.i.d.~realizations of an unknown multivariate Binomial family - see Section~\ref{Pb_Statement} 
- with unknwon distribution $q$, and given a known distribution $p$, we want to test 
$$\mathcal{H}_0: p=q~~~~~~~vs~~~~~~~\mathcal{H}_1: d(p,q) \geq \rho,$$
for a given distance $d$ and separation radius $\rho$. }\\

The difficulty of this testing problem is characterized by the minimal separation radius $\rho$ needed to ensure the existence of a test that is uniformly consistent under both the null and the alternative hypothesis - i.e.~a test whose worst-case error is smaller than a given $\eta>0$, and to identify such a test. 
See Section~\ref{Pb_Statement} for a precise definition of the setting.\\

In this paper, we will mostly focus on the following goals:
\begin{itemize}
\item We focus on the case where the distance $d$ is the $\ell_t$ distance, namely, if $p = (p_1,\dots,p_N)$ and $q = (q_1,\dots,q_N)$, then $d(p,q) = \Big(\sum\limits_{i=1}^N |q_i - p_i|^t\Big)^{1/t}$ for any $t\in [1,2]$. Typically, the case $t=2$ and $t=1$ (total variation distance for discrete distributions) are considered, and we interpolate between these two extreme cases.
    \item Our main objective will be to develop tests - as well as matching lower bounds - 
    for this identity testing problem that are \textit{locally optimal} in that the minimax separation distance $\rho$ should depend tightly on $p$. 
    Indeed, it is clear that some $p$ will be ``easier'' to test than others. 
    Consider e.g.~the following two extreme cases in the case of discrete (multinomial) distributions over $\{1,\dots,N\}$: 
    (i)~the very ``easy'' case where $p$ is a Dirac distribution on one of the coordinates, which implies a very low noise, and 
    (ii) the very ``difficult'' case where all entries of $p$ are equal to $1/N$, which maximizes the noise. 
    It is clear that the minimax local separation distance should differ between these two cases and be much smaller in case (i) than in case (ii). 
    We aim at studying the minimax local separation distance for any $p$, and characterize tightly its shape depending on $p$.
\end{itemize}

The existing literature about hypothesis testing~\cite{neyman1933ix} is profuse: the goodness-of-fit problem has been thoroughly studied, especially in the case of signal detection in the Gaussian setting, notably by Ingster - see~\cite{ingster2012nonparametric} - and has given rise to a vast literature. In parallel to the study of hypothesis testing, there exists a broad literature on the related problem of property testing with seminal papers such as \cite{rubinfeld1996robust,goldreich1998property}.\\

The identity testing problem in multinomials - i.e. probability distributions over a finite set - has been widely studied in the literature. We refer the reader to~\cite{CanonneTopicsDT2022}, \cite{canonne2020survey}, \cite{balakrishnan2018hypothesis} for excellent surveys.
When observing $n$ iid data with unknown discrete distribution $q$ and when fixing a distribution $p$, the aim is to derive the minimal separation distance $\rho$ so that a uniformly consistent test exists for testing $\mathcal{H}_0: p=q \text{ vs } \mathcal{H}_1(\rho): d(p,q)\geq \rho$. 
Note that this problem is also often considered in the dual setting of \textit{sample complexity}, where the goal is to find the minimal number of samples $n$ such that a consistent test exists for a given separation $\rho>0$. 
One distinguishes between \textit{global results} which are obtained for the worst case of the distribution $p$, and \textit{local results}, where the minimax separation distance is required to depend precisely on any given $p$. For global results, see e.g.~\cite{ingster1984asymptotically} (in Russian), \cite{ingster1986minimax}, \cite{ermakov1995minimax}, \cite{gine2016mathematical}, \cite{paninski2008coincidence}, and also in the related two-sample testing problem - where both $p,q$ are unknown and observed through samples - see e.g.~\cite{batu2000testing,chan2014optimal}. 
In the present paper, we focus on \textit{local} results. 
In the case of the $\ell_1$ distance, important contributions to local testing have been established in e.g.~\cite{valiant2017automatic}, \cite{diakonikolas2016new}. 
Note that these papers provide results in terms of sample complexity, and more recently, the paper~\cite{balakrishnan2019hypothesis} has re-considered this problem in terms of minimax separation distance - focusing also on the case of smooth densities. 
Another quite related work is \cite{berrett2020locally}, investigating the rate of goodness-of-fit testing in the multinomial case, in the $\ell_1$ and $\ell_2$ distances, under privacy constraints. 
Regarding the related two sample testing problem, see~\cite{acharya2012competitive, bhattacharya2015testing, diakonikolas2016new, kim2018robust}. 
This multinomial framework proves very useful for a wide range of applications, which include Ising models \cite{daskalakis2019testing}, bayesian networks \cite{canonne2020testing} or even quantum mechanics \cite{buadescu2019quantum}.\\

The papers~\cite{valiant2017automatic, balakrishnan2019hypothesis} are the most related to our present results, due to the equivalence between the multivariate binomial and Poisson distribution settings and the multinomial setting after a Poissonization trick  - see section \ref{ss:mult} for more details on why our setting encompasses those settings. 
We postpone a precise discussion between our result and this stream of literature to the core of the paper\footnote{We compare with this stream of literature under our upper and lower bounds in Sections~\ref{Results}, and also in the discussion in Section~\ref{Discussion}.}, since it is technical. As high-level comments, we restrict to remarking this stream of literature only considers separation in total variation distance, namely the $\ell_1$ distance for discrete distributions. \\

Note that goodness-of-fit testing for inhomogeneous Erdös-Rényi random graphs (see the definition e.g. in~\cite{ghoshdastidar2020two}), is a direct an important corollary of our result about multivariate binomial local testing. 
This result is therefore interesting as only little literature exists about identity testing in random graphs - and to the best of our knowledge, no literature exists about \textit{local} identity testing in the sense described above (see for example~\cite{dan2020goodness} for global testing in inhomogeneous random graphs). 
In recent machine learning and statistical applications, the increasing use of networks has made large random graphs a decisive field of interest. To name a few topics, let us mention community detection, especially in the stochastic block model (\cite{abbe2017community}, \cite{arias2014community}, \cite{verzelen2015community}, \cite{abbe2016achieving}, \cite{decelle2011asymptotic}), in social networks (\cite{bedi2016community}, \cite{wang2015community}), as well as network modeling (\cite{albert2002statistical}, \cite{lovasz2012large}), or network dynamics (\cite{berger2005spread}). 
The papers~\cite{ghoshdastidar2020two} and \cite{ghoshdastidar2018practical} propose an analysis of the two sample case, under sparsity: Given two populations of mutually independent random graphs, each population being drawn respectively from the distributions $P$ and $Q$, they perform the minimax hypothesis testing $\mathcal{H}_0: P=Q$ vs $\mathcal{H}_1: d(P,Q) \geq \rho$ for a variety of distances $d$, and identify optimal tests over the classes of sparse graphs that they consider. The paper \cite{qian2014efficient} identifies a computationally efficient algorithm for testing the separability of two hypotheses. Testing between a stochastic block model versus an Erdös-Rényi model has been studied in \cite{gao2017testing} and \cite{lei2016goodness}. 
Phase transitions are also known for detecting strongly connected groups or high dimensional geometry in large random graphs (\cite{bubeck2016testing}). The paper \cite{tang2017semiparametric} tests random dot-product graphs in the two sample setting with low-rank adjacency matrices. The paper~\cite{ghoshdastidar2017two} examines a more general case in which the graphs are not necessarily defined on the same set of vertices. 
To summarize, only few papers address the construction of efficient tests in random graphs - although this would be valuable in various areas such as social networks \cite{moreno2013network}, brain or `omics’ networks \cite{ginestet2017hypothesis} \cite{hyduke2013analysis}, testing chemicals \cite{shervashidze2011weisfeiler} or ecology and evolution \cite{croft2011hypothesis}. Moreover, and to the best of our knowledge, no paper considers the \textit{local version} of the testing problem - i.e. focuses on obtaining separation distances that depend on the null hypothesis.\\

The paper is organized as follows: In Section \ref{Pb_Statement}, we describe the setting by defining the multivariate binomial model and the minimax framework. In Section \ref{Results}, state our main theorem, which gives an explicit expression of the minimax separation radius as a function of $p$ and $n$. In Section \ref{ss:mult}, we establish the equivalence between the binomial, the Poisson and the multinomial settings. In Section \ref{Discussion}, we discuss our results, by comparing them with the state of the art, especially with the multinomial setting.
In Section~\ref{sec:Lower_bounds}, we describe our lower bound construction.
In Section~\ref{sec:Upper_bounds}, we describe our tests and state theoretical results guaranteeing their optimality.
We finally provide additional comments on our results in Section~\ref{sec:further_remarks}. All proofs are deferred to the Appendix. 

\section{Problem statement}\label{Pb_Statement}

\vspace{-10mm}

\revision{
\subsection{Setting}\label{ss:not}}

\revision{
We first introduce the Binomial setting. In Section~\ref{ss:mult}, we will introduce two other very related settings (the Multinomial and the Poisson settings) and prove that the associated minimax rates can be deduced from the Binomial case.}\\

\revision{Let $N \in \mathbb N$, $N \geq 2$ and define $\mathcal{P}_N = [0,1]^N.$
Let $q=(q_1, \dots, q_N) \in \mathcal{P}_N$ be an unknown vector of Bernoulli parameters. Assume that we observe $X_1, \dots, X_n$ iid such that each $X_i$  can be written as $X_i = \big(X_i(1), \dots, X_i(N)\big)$ where all of the entries $X_i(1), \dots, X_i(N)$ are mutually independent and $X_i(j) \sim Ber(q_j)$. We slightly abuse notation and write $X_1, \dots, X_n \overset{iid}{\sim} q$ when $X_1,\dots, X_n$ are generated with this distribution. Assume that $n$ is even: $n=2k$, for $k \in \mathbb{N}$. This assumption can be made \textit{wlog} and makes the analysis of the upper bound more convenient by allowing for sample splitting. We denote the total variation distance between two probability measures by $d_{TV}$ and for any $p \in \R^N$ and for $t > 0$,  we define 
\begin{align*}
    \| \hspace{.3mm} p \hspace{.3mm} \|_t & = \bigg[\sum_{j=1}^N |p_j|^t\bigg]^{1/t}.
\end{align*}}

\subsection{Minimax Testing Problem}\label{ss:min}


We now define the testing problem considered in the paper. Let $\eta \in (0,1)$ be a fixed constant and let $t\in [1,2]$. 
We are given a known vector $p \in  \mathcal{P}_N$ and we suppose that the data is generated from an unknown vector $q$: $X_1,\dots, X_n \overset{iid}{\sim} q$. We are interested in the following testing problem:\\
\begin{equation}\label{Test}
    \mathcal{H}_0^{p} : q=p \hspace{10mm} \text{vs} \hspace{10mm} \mathcal{H}_1^{\rho, p,t}: q \in \mathcal{P}_N; \; \|p-q\|_t \geq \rho.
\end{equation}\\
This problem is called ``goodness-of-fit testing problem''. When no ambiguity arises, we write $\mathcal{H}_0$ and $\mathcal{H}_1$ to denote the null and alternative hypotheses.\\ 

A \textbf{\textit{test}} $\psi$ is a measurable function of the observations $X_1,\dots, X_n$, taking only the values $0$ or $1$. 
We measure the quality of any test $\psi$ by its \textbf{\textit{maximum risk}}, defined as:
\begin{align}
    R(\psi) &:= R_{\rho, p, t, n}(\psi)\nonumber\\ 
    &= \mathbb{P}_p(\psi = 1) + \sup_{\substack{q \text{ s.t. } \\ \|p-q\|_t \geq \rho}} \mathbb{P}_q(\psi = 0).\label{risk}
\end{align}

$R(\psi)$ is the sum of the type-I and the type-II errors.
\poubelle{
\begin{itemize}
    \item $ \mathbb P_P(\psi = 1) $ is the type-I error, i.e. the probability of deciding $\mathcal{H}_1$ when $\mathcal{H}_0$ holds.
    \item $\sup_{\|p-q\|_t \geq \rho} \mathbb{P}_q(\psi = 0)$ is the worst-case type-II error. Indeed, for $Q$ such that $\|p-q\|_t \geq \rho$,  $\mathbb{P}_q(\psi = 0)$ corresponds to the probability of deciding in favor of $\mathcal{H}_0$ when $\mathcal{H}_1$ holds. Therefore, $\sup_{d(P,Q) \geq \rho} \mathbb{P}_q(\psi = 0)$ is the type-II error of $\psi$ if the data were drawn from the most unfavorable distribution $Q \in \mathcal{H}_1$ for $\psi$.
\end{itemize}}

\hfill


The \textbf{\textit{minimax risk}} is the risk of the best possible test, if any:
\begin{align*}
    R^* & := R^*_{\rho, p, t, n} =  \inf_{\psi \text{ test}} R(\psi)\\
    & = \inf_{\psi \text{ test}} \left[\mathbb{P}_p(\psi = 1) + \sup_{Q: \|p-q\|_t\geq \rho} \mathbb{P}_q(\psi = 0)\right].
\end{align*}
Note that $R^*:= R^*_{\rho, p, t, n}$ depends on the choice of the norm indexed by~$t$, the vector $p$, the separation radius $\rho$, and the sample size $n$. Since all quantities depend on $p$, we say that the testing problem is \textit{local} - around $p$ - as opposed to classical approaches in the minimax testing literature, where one generally only considers a family of vectors $p$ and focuses only on the worst case results over this family - see e.g.~\cite{ghoshdastidar2017two}. \\

In the following, we fix an absolute constant $\eta \in (0,1)$ and \textbf{we are interested in finding the smallest $\rho^*_{p,t,n}$ such that $R^*_{\rho^*_{p,t,n}, p,t,n} \leq \eta$}: 
\begin{equation}\label{eq:sepdist}
    \rho^*_{p,t,n}(\eta) = \inf \left\{ \rho >0 \; : \; R^*_{\rho, p,t,n} \leq \eta \right\}.
\end{equation}
We call $\rho^*_{p,t,n}(\eta)$ the \textit{$\eta$-minimax separation radius}. Whenever no ambiguity arises, we drop the indexation in $n, p, t,\eta$ and write simply $\rho^*, R^*_\rho, R_\rho(\psi)$ - but these variables remain important, as will appear later on.\\




The aim of the paper is to give the explicit expression of $\rho^*_{p,t,n}$ up to constant factors depending only on $\eta$ and to construct optimal tests, for any $p \in \mathcal{P}_N$ and all $t \in [1,2]$. 



\paragraph{Additional notation.}

Let $\eta>0$. For $f$ and $g$ two real-valued functions defined, we say that $f \lesssim_\eta g$ (resp. $f \gtrsim_\eta g$) if there exists a constant $c_\eta >0$ (resp. $C_\eta >0$) depending only on $\eta$, such that $c_\eta g \leq f$ (resp. $f \geq C_\eta g$). We write $f \asymp_\eta g$ if $g \lesssim_\eta f \text{ and } f \lesssim_\eta g$. Whenever the constants are absolute, we drop the index $\eta$ and just write $\lesssim, \gtrsim, \asymp$. We respectively denote by $x \vee y$ and $x \wedge y$ the maximum and minimum of the two real values $x$ and $y$. 




\section{Results}\label{Results}

Without loss of generality, assume that $\max\limits_{1\leq j \leq N} p_j \leq \frac{1}{2}$. Otherwise, if for some $j \in \{1,\dots, N\}, \; p_j > \frac{1}{2}$, replace $p_j$ by $1 - p_j$ and replace accordingly $X_i(j)$ by $1-X_i(j)$ for all $i = 1, \cdots, n = 2k$. 
\textit{Wlog}, assume that all entries of the known vector $p$ are sorted in decreasing order: 
$$p = (p_1\geq p_2 \geq \cdots \geq p_N).$$

For any index $1 \leq u \leq N$, we define the vectors
\begin{align*}
    \begin{cases}
    p_{\leq u} = (p_1, \cdots, p_u, 0, \cdots, 0)\\
    p_{> u} = (0, \cdots, 0, p_{u+1}, \cdots, p_N).
    \end{cases}
\end{align*}

Let $\eta>0$. In what follows, we write 
\begin{equation}\label{def_r_b}
    r = \frac{2t}{4-t} \hspace{5mm} \text{ and } \hspace{5mm} b = \frac{4-2t}{4-t}. 
\end{equation}
for $p$ we also define 
\begin{equation}\label{def:def_I}
    I = \min \left\{J : \sum_{i>J} p_i^2 \leq \frac{\cI}{n^2}\right\}
\end{equation}
where $\cI$ is a small enough constant depending only on $\eta$. We will prove the following theorem.
\revision{\begin{theorem}\label{Rate} For all $t \in [1,2]$, the following bound holds, up to a constant depending only on $\eta$ and $t$:
$$\rho^* \asymp_{\eta,t} \sqrt{\frac{\big\|p_{\leq I}\big\|_r}{n}} +  \frac{\big\|p_{>I}\big\|_1^{\frac{2-t}{t}}}{n^\frac{2t-2}{t}}  + \frac{1}{n},$$
where we recall that $I=I(n,p,t)$.
\end{theorem}}

The lower bounds and the minimax test are given in Section~\ref{sec:Lower_bounds} and Section~\ref{sec:Upper_bounds}.

\subsection{Equivalence between the Binomial, the multinomial and the Poisson setting}\label{ss:mult}

We now move to the multinomial and Poisson settings. In the following propositions, we state that the multinomial and the multivariate Binomial model are equivalent to the multivariate Poisson setting after using the \textit{Poissonization trick}, and that the results from the binomial setting can be transferred to the other two settings. The Poissonization trick consists in drawing $\widetilde n \sim Poi(n)$ observations instead of $n$, either from the multinomial or from the multivariate binomial model. The resulting data is exactly distributed as a multivariate Poisson family.
\begin{prop}[Poissonization trick for multinomials]\label{Poissonization_trick_mult}
Let $n \geq 2$ and assume that $p,q$ are probability vectors, i.e.~such that $\sum_i p_i = \sum_i q_i =1$. Let $\widetilde n \sim Poi(n)$. Conditional on $\widetilde n$, let $Z_1, \cdots, Z_{\widetilde n} \overset{iid}{\sim} \mathcal{M}(q)$. We build the histogram sufficient statistic by defining, for all $j=1, \cdots, N$,  $H_j = \sum_{i=1}^{\widetilde n} \mathbb{1}\left\{ Z_i = j \right\}$. Then for all $j$, $H_j \sim Poi(nq_j)$ and $H_1, \cdots, H_N$ are mutually independent.
\end{prop}

\begin{prop}[Poissonization trick for binomial families]\label{Poissonization_trick_bin}
Let $n \geq 2$ and $\widetilde n \sim Poi(n)$. Conditional on $\widetilde n$, let $X_1, \cdots, X_{\widetilde n} \overset{iid}{\sim} \bigotimes_{j=1}^N Ber(p_j)$. Then $\sum_{i=1}^{\widetilde n} X_i \sim \bigotimes_{j=1}^N Poi(np_j)$.
\end{prop}
These two propositions are classical and follow from basic properties of the Poisson, Multinomial, and Binomial distributions. We rewrite them here only to provide some context on the equivalences that follow.

\hfill

Without loss of generality, assume that $p_1 \geq \cdots \geq p_N$. We consider the following settings:

\begin{enumerate}
    \item \textbf{Binomial case:} This is the setting considered above. We define $\mathcal{P}^{(Bin)} = \{Ber(p); \; p \in \R_+^N \}$ where by convention, $Ber(p) := \bigotimes_{j=1}^N Ber(p_j)$. We fix $p\in \mathcal{P}^{(Bin)}$ and suppose we observe $X_1, \cdots, X_n \overset{iid}{\sim} Ber(q)$ for $q \in \mathcal{P}^{(Bin)}$ unknown. We consider the \textbf{binomial} testing problem:
    \begin{align*}\label{testing_pb_binomial}
    H_0^{(Bin)}: q=p ~~\text{ vs }~~ H_1^{(Bin)}: \begin{cases}q\in \mathcal{P}^{(Bin)}; \\ \|q-p\|_t \geq \rho.\end{cases}  
    \end{align*}

    \item \textbf{Poisson case:} $\mathcal{P}^{(Poi)} = \{Poi(p); \; p \in \R_+^N \}$ where by convention, $ Poi(p) := \bigotimes_{j=1}^N Poi(p_j)$. We fix $p\in \mathcal{P}^{(Poi)}$ and suppose we observe $Y_1, \cdots, Y_n \overset{iid}{\sim} Poi(q)$ for $q \in \mathcal{P}^{(Poi)}$ unknown. We consider the \textbf{Poisson} testing problem:
    \begin{align*}
    H_0^{(Poi)}: q=p ~~\text{ vs }~~ H_1^{(Poi)}: \begin{cases}q\in \mathcal{P}^{(Poi)}; \\ \|q-p\|_t \geq \rho.\end{cases}  
    \end{align*}
    
    \item \textbf{Multinomial case} $\mathcal{P}^{(Mult)} = \big\{\mathcal{M}(p) \big| \; p \in \R_+^N, \sum_{j=1}^N p_j =1 \big\}$ where $ \mathcal{M}(p)$ denotes the multinomial distribution over $\{1,\dots, N\}$. We fix $p\in \mathcal{P}^{(Mult)}$ and suppose we observe $Z_1, \cdots, Z_n \overset{iid}{\sim} \mathcal{M}(q)$ for $q \in \mathcal{P}^{(Mult)}$ unknown. We consider the \textbf{Multinomial} testing problem:
    \begin{align*}
    H_0^{(Mult)}: q=p ~~\text{ vs }~~ H_1^{(Mult)}: \begin{cases}q\in \mathcal{P}^{(Mult)}; \\ \|q-p\|_{\mathcal{M},t} \geq \rho.\end{cases}  
    \end{align*}
where for $x = (x_1, \cdots, x_N)$: $\|x\|_{\mathcal{M},t} = \left[\sum_{j=2}^N |x_j|^t\right]^{1/t} $ is the multinomial norm, defined without taking the first coordinate into account. Indeed, because of the shape constraint $\sum p_j = 1$, the first coordinate does not bring any information and can be deduced from the $N-1$ coordinates.

\end{enumerate}
For these three testing problems, we define respectively $\rho_{Bin}^*(n,p,t,\eta),\\ \rho_{Poi}^*(n,p,t,\eta), \rho_{Mult}^*(n,p,t,\eta)$ for the minimax separation distances in the sense of Equation~\eqref{eq:sepdist}, for each of the testing problems.

We state the following statement regarding the equivalence between all models.
\begin{lemma}\label{bin_eq_poi}\normalfont{\textbf{(Equivalence between the Binomial and Poisson settings)}}
Let $t\in [1,2]$. There exist two absolute constants $c_{BP}, \; C_{BP} >0$ depending on $\eta$ such that $\forall p \in [0,1]^N, \; \forall n \geq 2 \; \eta>0, \;:$ 
$$ c_{BP}\;\rho^*_{Bin}(n,p, t,\eta) \leq \rho^*_{Poi}(n,p) \leq C_{BP}\;\rho^*_{Bin}(n,p, t,\eta).$$
\end{lemma}

\vspace{3mm} 

\begin{lemma}\label{mult_eq_poi}\normalfont{\textbf{(Equivalence between Multinomial and Poisson settings)}}
Let $t\in [1,2]$. It holds that $\forall p \in [0,1]^N, \; \forall n \geq 2\; \eta>0$, if $\sum_{i=1}^N p_i = 1$: 
$$\;\rho^*_{Mult}(n,p,t,\eta) \lesssim_\eta \rho^*_{Poi}(n,p^{-\max}) \lesssim_\eta \;\rho^*_{Mult}(n,p,t,\eta)$$ where $p^{-\max} := (p_2, \cdots, p_N)$.
\end{lemma}








\hfill

This entails the following corollary regarding the minimax rates of testing in the multinomial model:

\revision{\begin{corollary}\label{cor:multpois}
Let $t\in [1,2]$. The minimax separation radii in the Poisson and multinomial cases are respectively given by:
\begin{align*}
    \rho_{Poi}^*(n,p,t,\eta) &\asymp_\eta \sqrt{\frac{\norm{p_{\leq I}}_r}{n}} + \frac{\norm{p_{>I}}_1^{\frac{2-t}{t}}}{n^\frac{2t-2}{t}}  + \frac{1}{n}  ~~~~~~~ \text{ for } p\in \mathcal{P}^{(Poi)}\\
    \rho_{Mult}^*(n,p,t,\eta) &\asymp_\eta \sqrt{\frac{\norm{p^{-\max}_{\leq I}}_r}{n}} + \frac{\norm{p_{>I}}_1^{\frac{2-t}{t}}}{n^\frac{2t-2}{t}}  + \frac{1}{n} ~~~~ \text{ for } p\in \mathcal{P}^{(Mult)},
\end{align*}
where we recall that $I=I(n,p,t)$.
\end{corollary}}
Note that the upper bounds in the Poisson model are obtained using our tests on the Poisson vector, and the upper bounds in the Multinomial model are obtained using our tests on the last $N-1$ coordinates of the estimates of probabilities of each categories.

\section{Discussion}\label{Discussion}









%

In this entire section, we mostly discuss the Multinomial setting - whose rates are given in Corollary~\ref{cor:multpois} - which is the most studied setting in the literature. To alleviate notations, we will write $ \rho^*(n,p)$ for the minimax separation distance in the Multinomial model, dropping the dependence on $\eta$. 

\subsection{Locality of the results}

In the present paper, we derive sharp local minimax rates of testing in the binomial, Poisson and multinomial settings. The locality property is a major aspect of the results: for each fixed $p$ 
we identify the detection threshold \textit{associated to $p$}, where $p$ is allowed to be any distribution in the class. 
For related local results in the case of the $\ell_1$ or $\ell_2$ norm, see e.g. \cite{valiant2017automatic}, \cite{diakonikolas2016new}, \cite{balakrishnan2019hypothesis} \cite{berrett2020locally}. This approach is less standard than the usual \textit{global} approach, which consists in finding the largest detection threshold in the class, i.e.~for the \textit{worst case} of $p$ - see e.g.  \cite{ingster1984asymptotically} (in Russian), \cite{ingster1986minimax}, \cite{ermakov1995minimax}, \cite{gine2016mathematical}, \cite{paninski2008coincidence}. 
Yet, local results can substantially improve global results: for instance, in the multinomial case and for the $\ell_2$ norm, the global separation radius for an $N$-dimensional multinomial is classically $N^{-1/4}/\sqrt{n}$, and is reached in the case where $p$ is uniform distribution. 
However, if $p = (1,0, \dots, 0)$ is a Dirac multinomial, then from our results the rate of testing in $\ell_2$ norm is $\frac{1}{n}$, hence much faster than the global rate. 
Even for fixed $N$, one can actually find a sequence of null distributions $p^{(n)}$ whose associated separation distance $\rho^*_{Mult}(n,p^{(n)},2,\eta)$ reaches any rate $1/n^\alpha$ for any $1/2 \leq \alpha \leq 1$ 
This consequently improves the global rate even for less extreme discrete distributions than Dirac multinomials. To give an example, consider an exponentially decreasing multinomial distribution
$p^{(n)} = \big(\frac{c}{n^{(2\alpha -1)j}}\big)_{j=1}^N$ for the renormalizing constant $c = n^{2\alpha-1} \frac{1-1/n^{2\alpha-1}}{1-1/n^{(2\alpha-1)N}} \asymp n^{2\alpha - 1}.$
Then, evaluating the local rate in $\ell_2$ (allowing us to consider the whole set of coefficients as the bulk, see Section \ref{Influence_of_the_norm} below), we get:
$$ \rho_{Mult}^*(n,p^{(n)}, 2,\eta) \asymp_\eta \sqrt{ \frac{\left\|p^{-\max}\right\|_2}{n}} + \frac{1}{n} \asymp_\eta \frac{1}{n^\alpha}.$$

\subsection{Comparison with existing literature in the multinomial case}\label{ss:valiant}

Our results are quite related to those of \cite{valiant2017automatic}, which examines the multinomial testing problem for the $\ell_1$ distance and in terms of sample complexity. More precisely, for a fixed $N$-dimensional multinomial distribution $p$, and for a fixed separation $\rho$, this work investigates the smallest number $n^*(p, \rho)$ of samples $X_1, \cdots, X_n \overset{iid}{\sim} \mathcal{M}(p)$ needed to ensure that the Multinomial  testing problem introduced in Section~\ref{ss:mult} has a minimax risk less than $2/3$, for a fixed separation distance $\rho>0$.  Formally this is defined as $n^*(p,\rho) = \min \left\{n \in \mathbb N: R^*_{\rho, p, t, n} \leq 2/3\right\}$ where $R^*_{\rho, p, t, n}$ denotes here the minimax risk for the multinomial problem\footnote{See Equation~\eqref{risk} for the definition of this quantity in the graph problem.}. Note that the quantities $n^*$ and $\rho^*$ are dual, for $\eta = 2/3$.\\

\cite{valiant2017automatic} proves the following bounds to characterize the optimal sample complexity $n^*(p,\epsilon)$ when given a fixed $\epsilon>0$:

$$ \frac{1}{\epsilon} + \frac{\norm{p_{-\epsilon}^{-\max}}_{2/3}}{\epsilon} \;\;\lesssim\;\; n^*(p,\epsilon) \;\;\lesssim\;\;\frac{1}{\epsilon} + \frac{\norm{p_{-\epsilon/16}^{-\max}}_{2/3}}{\epsilon}.$$
In the above bound, $p = (p_1, \cdots, p_N)$ where $p_1 \geq \cdots \geq p_N \geq 0$ and $\sum_{i=1}^n p_i = 1$. For $\epsilon > 0$, let $J$ be the smallest index such that $\sum_{i>J} p_i \leq \epsilon$. The notation $p_{-\epsilon}^{-\max}$ denotes $(p_2, \dots, p_J)$. \\

We generalize the result in several respects:
\begin{itemize}
    \item We consider the whole range of $\ell_t$ distances for $t$ in the segment $[1,2]$ and characterize the \textbf{\textit{local}} rates of testing in each case,
    \item We generalize the multinomial case to the graph case (binomial case) and to the Poisson setting, through the Poissonization trick.
\end{itemize}

\revision{In Appendix~\ref{sec:Wasserman_tight}, we justify that the upper and lower bounds from~\cite{valiant2017automatic}, when translated in terms of separation radius as in~\cite{balakrishnan2019hypothesis} actually match in the multinomial case, although claimed otherwise by the authors of~\cite{balakrishnan2019hypothesis} themselves. It was therefore unclear in the literature so far that matching upper and lower bounds on the critical radius were actually known in the case $t=1$.}
All of these cases involve the following ideas. The distribution can be split into bulk (set of large coefficients, with a subgaussian phenomenon) and  tail (set of small coefficients, with a subpoissonian phenomenon). To the best of our knowledge, the way we define the tail is new. It allows us to establish a clear cut-off between these two optimal sets, fundamentally differing through the behavior of the second order moment of $p$.\\ 

\revision{
The present paper can be linked with~\cite{blais2019distribution}, which considers instance optimal identity testing. Specifically, \cite{blais2019distribution} obtains a different characterization of the sample complexity for the case $t=1$, in terms of a fundamental quantity in the theory of interpolation of Banach spaces, known
as Peetre’s $K$-functional. This functional is defined for all $u>0$ as
$$ \kappa_p(u) = \inf_{p' + p'' = p} \|p'\|_1 + u \|p''\|_2.$$
This paper proves that for fixed $\epsilon \in (0,1)$, any test for testing identity to $p$ needs at least $\Omega(\kappa_p^{-1})(1-2\epsilon)$ samples in order to have a risk less than $\eta$. 
In Section 6.3, especially equation $(14)$ this paper discusses the non-tightness of~\cite{valiant2017automatic}. 
Note that their bound is not optimal either, but is incomparable to~\cite{valiant2017automatic}. 
This paper also provides a testing algorithm considering separately tail and heavy elements of the distribution, as well as a lower bound that uses interpolation theory to divide the problem into two types of elements - the $\ell_1$ contribution (heavy elements) and the $\ell_2$ ones (uniform-like).}\\

\revision{Building on this work, ~\cite{9211522} Appendix D: provides a general reduction scheme showing how to perform instance-optimal one-sample testing, given a "regular" (non-instance optimal) one-sample testing algorithm (even only for uniformity testing). This applies in particular to local privacy, or testing under communication constraints, or even without constraints at all.
}

\section{Lower bounds}\label{sec:Lower_bounds}
We recall the definitions of $r$ and $b$ in equation~\eqref{def_r_b}.
In what follows, index $A$ is defined as 
\begin{equation}\label{def_A}
    A= A_{P,t,n}(\eta) := \max\Bigg\{ a \leq I: p_a^{b/2} \geq \frac{\cA}{\sqrt{n}\big(\sum\limits_{i\leq I} p_i^r \big)^\frac{1}{4}} \Bigg\},
\end{equation}
where $\cA>0$ is a small enough constant depending only on $\eta$. \revision{We adopt the convention that $\max \emptyset = -\infty$ and that $p_{\leq -\infty} = \emptyset$ and $p_{>-\infty} = p$.} We start by presenting the lower bound part of Theorem \ref{Rate}. We divide the analysis into two parts: a lower bound for the large coefficients of $p$ (bulk) and  a lower bound for the small coefficients of $p$ (tail). The bulk will be defined as the set $p_{\leq A}$ and the tail as $p_{>A}$.

\subsubsection{Lower bound for the bulk}

To prove the lower bound, we identify a radius $\rho$ such that, if the $\ell_t$ distance between $\mathcal{H}_0$ and $\mathcal{H}_1$ is less than $\rho$, then any test has risk at least $\eta$. Therefore, by definition of $\rho^*$, $\rho$ is necessarily a lower bound on $\rho^*$. 
\begin{proposition}\label{LB:Bulk} 

Let $t \in [1,2]$. There exists a constant $\uc'>0$ depending only on $\eta$, as well as a distribution $q$ such that for any test $\psi$ we have
$$\big\|(q-p)_{\leq A}\big\|_t \geq \uc'\left( \frac{\left\|p_{\leq A}\right\|_r^{r/t}}{\sqrt{n}\,\left\|p_{\leq I}\right\|_r^{r/4}} + \frac{1}{n}\right),$$
and 
$$\mathbb P_p(\psi = 1) + \mathbb P_q(\psi = 0) \geq \eta.$$
\end{proposition}

This implies that $\rho = \frac{\left\|p_{\leq A}\right\|_r^{r/t}}{\sqrt{n}\,\left\|p_{\leq I}\right\|_r^{r/4}} + \frac{1}{n}$ is a lower bound on the minimax separation radius $\rho^*$. \\

Note that the lower bound in $\frac{1}{n}$ is trivial since changing any entry of $p$ by $\frac{1}{n}$ is not detectable with high probability. Now let us examine the first part of the rate. To prove this lower bound, we use Le Cam's two points method by defining a prior distribution over a discrete subset of $\mathcal{P}_N$ satisfying $\mathcal{H}_1$. More precisely, for all $\left(\delta_1, \cdots, \delta_A\right) \in \left\{\pm 1 \right\}^A$ we define the distribution $q_\delta$ such that:
\begin{equation}\label{Prior_bulk}
    \left(q_\delta\right)_j =\left\{\begin{array}{ll}
      p_j + \delta_i \gamma_j & \text{if } j\leq A \\
     p_j & \text{ otherwise,}  
\end{array} \right. 
\end{equation}
where, for some small enough constant $\cgamma>0$ depending only on $\eta$:
\begin{equation}\label{gamma}
\gamma_i = \frac{\cgamma \, p_i^\frac{2}{4-t}}{\sqrt{n} \left(\sum_{i \leq I} p_i^r\right)^\frac{1}{4}}.   
\end{equation}

\revision{The mixture $$ \bar{\mathbb P}_{\mathrm{bulk}} = \frac{1}{2^A} \sum_{\delta \in \{\pm 1\}^A} q_\delta^{\otimes n}$$
defines a probability distribution over the set of observations $X_1,\dots X_n$, such that, conditional on $\delta \in \{\pm 1\}^A$, the observations are iid with probability distribution $q_\delta$.}\\

The core of the proof is to prove that observations $X_1,\dots, X_n$ drawn from this mixture distribution $\bar{\mathbb P}_{\mathrm{bulk}}$ are so difficult to distinguish from observations $X'_1,\dots X'_n$ drawn from $\mathbb P_p$, that the risk of any test is necessarily larger than $\eta$. 
This brings us to the conclusion of our proposition since any distribution $q_\delta$ is separated away from $p$ by an $\ell_t$ distance equal to $\left(\sum_{i=1}^A \gamma_i^t\right)^\frac{1}{t} \asymp \frac{\left\|p_{\leq A}\right\|_r^{r/t}}{\sqrt{n}\left\|p_{\leq I}\right\|_r^{r/4}}$. Therefore, $\frac{\left\|p_{\leq A}\right\|_r^{r/t}}{\sqrt{n}\left\|p_{\leq I}\right\|_r^{r/4}}$ is necessarily a lower bound on the separation radius $\rho^*$
This lower bound is an extension to the case where $t\in [1,2]$ of the lower bound in~\cite{valiant2017automatic} which is given for the case $t=1$, up to some issues that are discussed in details in Subsection~\ref{ss:valiant}.

\subsubsection{Lower bound for the tail}

We now derive a lower bound for the tail $p_{>A}$, containing the smallest coefficients of $p$. The tail lower bound involves very different phenomena compared to the above bulk lower bound. The reason is that the definition of $A$ implies that on the tail, \textit{whp}, no same coordinate is observed twice or more among the $n$ data.

\begin{proposition}\label{LBtail} 

Let $t \in [1,2]$, and consider any test $\psi$. There exists a constant $\uc'>0$ depending only on $\eta$ and a distribution $Q$ such that 
$$\left\|(q-p)_{> A}\right\|_t \geq \uc' \frac{\left\|p_{>I}\right\|_1^{\frac{2-t}{t}}}{n^\frac{2t-2}{t}} ,$$
and 
$$\mathbb P_p(\psi = 1) + \mathbb P_q(\psi = 0) \geq \eta.$$
\end{proposition}


To prove this lower bound, we  once more use Le Cam's two points method with a \textit{sparse} prior distribution. Define the smallest index $U > I$ such that $n^2 p_U \norm{P_{\geq U}}_1 \leq c_u < 1$ where $c_u>0$ is a small constant defined in the appendix. We define 
$$ \bar \pi = \frac{c_u}{ n^2\norm{p_{\geq U}}_1} \text{ and } \pi_i = \frac{p_i}{\bar \pi}.$$

Index $U$ has no further meaning than to guarantee that for all $i\geq U$ : $\pi_i \in [0,1]$. In particular, $\pi_i$ is a Bernoulli parameter. Now, we define the following prior on $q$. For any $i<U$ we set $q_i = p_i$. Otherwise for $i \geq U$, we set $b_i \sim Ber(\pi_i)$ mutually independent, and 
\begin{equation}\label{prior_tail}
    q_b(j) = b_j \bar \pi,
\end{equation}
We now consider the mixture of the probability distributions $q_b$:
$$ \bar{\mathbb P}_{\mathrm{tail}} = \sum_{b \in \{0,1\}^{\{U+1, \dots, N\}}} \bigg(\prod_{j>U} \pi_j^{b_j} (1-\pi_j)^{1-b_j}\bigg) q_b^{\otimes n}. $$
As above, we prove that the data $X_1,\dots, X_n$ drawn from this mixture $\bar{\mathbb P}_{\mathrm{tail}}$ is difficult to distinguish from the data $X_1', \dots, X_n'$ drawn from $\mathbb P_p$. 
Moreover, we show that with high probability, the $\ell_t$ distance between $\bar{\mathbb P}_{\mathrm{tail}}$ and $p$, is larger, up to an absolute constant than
\[ \scalebox{1.1}{$ \frac{\left\| p_{\geq U}\right\|_1^\frac{2-t}{t}}{n^\frac{2(t-1)}{t}}.$
} \]

Finally, to conclude the proof, we show in Lemma \ref{I_and_U} that 
\[ \scalebox{1.2}{$ \frac{\left\| p_{\geq U}\right\|_1^\frac{2-t}{t}}{n^\frac{2(t-1)}{t}} + \frac{1}{n} \asymp_\eta \frac{\left\|p_{>I}\right\|_1^\frac{2-t}{t}}{n^\frac{2(t-1)}{t}} + \frac{1}{n} $} \] 
in words, that we can replace $U$ by $I$. This lower bound departs significantly from the one in~\cite{valiant2017automatic} in the case $t=1$, which is significantly simpler than for $t>1$ for the tail coefficients.



\subsubsection{Combination of both lower bounds}

By combining Propositions~\ref{LB:Bulk} and~\ref{LBtail}, we obtain the following theorem.
\begin{theorem}\label{LB} 

Let $t \in [1,2]$, and consider any test $\psi$. There exists a constant $\uc'>0$ depending only on $\eta$ and a distribution $q$ such that 
$$\left\|Q-P\right\|_t \geq \uc'\left( \sqrt{\frac{\norm{p_{\leq I}}_r}{n}} + \frac{\norm{p_{>I}}_1^{\frac{2-t}{t}}}{n^\frac{2t-2}{t}}  + \frac{1}{n}\right),$$
and 
$$\mathbb P_p(\psi = 1) + \mathbb P_q(\psi = 0) \geq \eta.$$
\end{theorem}

\hfill

This theorem implies that 
$$ \rho^*  \gtrsim_\eta \sqrt{\frac{\norm{p_{\leq I}}_r}{n}} + \frac{\norm{p_{>I}}_1^{\frac{2-t}{t}}}{n^\frac{2t-2}{t}}  + \frac{1}{n},$$
which is  a lower bound on the separation radius $\rho^*$, up to a positive constant depending only on $\eta$.\\

Note that when combining Propositions \ref{LB:Bulk} and \ref{LBtail}, we do not get exactly the expression in Theorem \ref{LB}. We actually obtain:
$$ \rho^* \gtrsim_\eta \frac{\norm{P_{\leq A}}_r^{r/t}}{\sqrt{n}\norm{P_{\leq I}}_r^{r/4}} + \frac{\left\|p_{>I}\right\|_1^{\frac{2-t}{t}}}{n^\frac{2t-2}{t}}+ \frac{1}{n}.$$

We therefore need to show that this expression is equivalent to that in Theorem \ref{LB}. This is done by using  Lemma \ref{A_and_I}, which states that we can replace $\frac{\norm{P_{\leq A}}_r^\frac{r}{t}}{\sqrt{n}\norm{P_{\leq I}}_r^\frac{r}{4}}$ by $\sqrt{\frac{\norm{p_{\leq I}}_r}{n}}$ without changing the rate, i.e.
$$ \frac{\norm{P_{\leq A}}_r^\frac{r}{t}}{\sqrt{n}\norm{P_{\leq I}}_r^\frac{r}{4}} + \frac{\left\|p_{>I}\right\|_1^{\frac{2-t}{t}}}{n^\frac{2t-2}{t}} + \frac{1}{n} ~~~~ \asymp_\eta ~~~~ \sqrt{\frac{\norm{p_{\leq I}}_r}{n}} + \frac{\left\|p_{>I}\right\|_1^{\frac{2-t}{t}}}{n^\frac{2t-2}{t}} + \frac{1}{n}. $$


\hfill

\textbf{\underline{Remark on index $A$}:} As explained in \eqref{Prior_bulk}, the optimal prior is of the form $p_i \pm \gamma_i$ where $\gamma_i$ is proportional to $p_i^\frac{2}{4-t}$, according to Equation~\eqref{gamma}. Since $\frac{2}{4-t} \leq 1$, we can have $\gamma_i > p_i$ if $p_i$ is too small, so that it is impossible to set the optimal prior $p_i \pm \gamma_i$, since $p_i - \gamma_i$ has to be a Bernoulli parameter. The index $A$ is just the last index ensuring $p_A \geq \gamma_A$ so that our lower bound construction is well-defined.\\

\textbf{\underline{Remark on index $I$}:} Index $I$ defines the largest set of coefficients $p_{>I}$ such that, \textit{whp}, no coordinate $j >I$ is observed twice or more. This is exactly the interpretation of the relation $\sum_{j>I} n^2 p_j^2 \leq \cI$ for a small constant $\cI$. As shown in Lemma \ref{A_and_1_over_n}, it is important that the definition of $A$ also implies that $\sum_{j>A} n^2 p_j^2 \leq \cI + \cA^4$, which leads us to tune the constants $\cI$ and $\cA$ such that this sum is small. Therefore, on the actual tail $(p_{>A})$, no same coordinate will be observed twice \textit{whp} under $H_0$. This is the reason why the phenomena involved are different on the bulk and on the tail. On the bulk, many coordinates are observed at least twice, which allows us to build an estimator based on the \textit{dispersion} of the data around its mean, namely the renormalized $\chi^2$ estimator which is a modified estimator of the variance. Like in the classical gaussian signal detection setting, the optimal procedure for detecting whether or not the data is drawn from $p$ is to estimate the dispersion of the data. \\ 
On the tail, however, each coordinate is observed at most once, so that the \textit{dispersion} of the data cannot be estimated. On this set, we rather design a prior distribution which mimics the behavior of the null distribution, while being as separated from it as possible. More precisely, we impose that $whp$, no coordinate is observed twice, and such that coordinate-wise, the expected number of observations is equal to that under the null hypothesis $p$. In short, this prior is designed such that its first order moment is equal to that under the null and its second order moment is unobserved \textit{whp}. Under both of these constraints, we maximize the $\ell_t$ distance between the null hypothesis $p$ and the possible distributions composing the prior. When $t>1$, the result of this process is a prior that needs to be relatively sparse - which is significantly more involved than the case $t=1$ treated in~\cite{valiant2017automatic}.\\

\revision{\textbf{\underline{Remark on the lower bounds}:} The bulk lower bound is close to that of~\cite{valiant2017automatic}. The tail lower bound relies on a sparse prior that is an existing technique (for example in sparse testing, see~\cite{baraud2002non}, \cite{collier2017minimax}, \cite{kotekal2021minimax}) and is very different from the construction in~\cite{valiant2017automatic}. Handling the indices $I, A$ and $U$ require careful manipulations that we believe are new techniques.}
\section{Upper bounds}\label{sec:Upper_bounds}

We use sample splitting to define
$$S = \sum_{i=1}^k X_i,~~~\mathrm{and}~~~S' = \sum_{i=k+1}^{n} X_i,$$
We also write $$b = \frac{4-2t}{4-t}.$$

\subsubsection{Test for the bulk coefficients}

We now introduce the following test statistic on the bulk coefficients, i.e.~the coefficients with index smaller than $A$ :
\begin{equation}\label{test_bulk}
    T_{\bulk} = \sum_{i \leq A} \frac{1}{p_i^b} \left(\frac{S_i}{k} - p_i \right)\left(\frac{S'_i}{k} - p_i \right),
\end{equation}
which is a weighted $\chi^2$ statistic. We now define the test 
$$\psi_{\bulk} = \mathbf 1\left\{T_{\bulk} >  \frac{\uc}{n} \left\|p_{\leq A}\right\|_r^\frac{r}{2}\right\},$$
where $\uc = 4/\sqrt{\eta}$ is a large enough constant, depending only on $\eta$. We prove the following proposition regarding this statistic and the bulk of the vector $p$.
\begin{proposition}
There exists $\uc'>0$, such that the following holds. 
\begin{itemize}
    \item Type I error is bounded:
    $$\mathbb P_p(\psi_{\bulk} = 1) \leq \eta/2.$$
    \item Type II error is bounded: for any $q$ such that
$$\norm{q_{\leq A}}_t \geq \uc' \left( \sqrt{\frac{\norm{p_{\leq I}}_r}{n}} + \frac{1}{n}\right),$$
it holds that
    $$\mathbb P_q(\psi_{\bulk} = 0) \leq \eta/2.$$
\end{itemize}
\end{proposition}


For $t=1$, we get $r = \frac{2}{3}$, which is the norm identified in \cite{valiant2017automatic}. However, our setting is slightly different for three reasons. First, we consider multivariate binomial families rather than multinomials. Second, we consider separation distance for a fixed $n$ instead of sample complexity. Third, our result holds for any $t\in [1,2]$. However, in Subsection~\ref{ss:mult}, we prove that multivariate binomial and multinomial settings are related and that the rates can be transferred from our setting to the multinomial case.\\

Note that our cut-off is defined differently from that in \cite{valiant2017automatic}. In \cite{valiant2017automatic}, the cut-off $I'$ is the smallest index such that, for a fixed $\epsilon$: $\sum_{i>I'} p_i \leq \epsilon$. This definition therefore only involves the first order moment of the null distribution. 
In our setting, conversely, we define index $I$ using the second order moment of the null distribution, as the smallest index such that $\sum_{i>I} p_i^2 \leq \frac{\cI}{n^2}$.\\ 

The above result also generalizes the bound identified in \cite{valiant2017automatic}, by characterizing the testing rate for all $t \in [1,2]$ and sheds light on a duality between the $\ell_t$ and $\ell_r$ norms when $r = \frac{2t}{4-t}$. 

\subsubsection{Test for the tail coefficients}

\hfill

The tail test is a combination of two tests. We define the histogram of the data which is a sufficient statistic:
$$ \forall j>A, N_j:= \sum_{i=1}^n \mathbb{1}\{X_i = j\}$$

We first define the test $\psi_2$ which rejects $\mathcal H_0$ whenever one tail coordinate is observed twice. 
\begin{equation}\label{test_psi2}
    \psi_2 = \mathbb{1} \big\{\exists j>A: N_j \geq 2\big\}
\end{equation}

 We also define a statistic counting the number of observations on the tail, and the associated test, recalling that $\uc = 4/\sqrt{\eta}$:
\begin{equation}\label{test_psi1}
     T_{1} = \sum_{i > A} \frac{N_i}{n} - p_i, ~~ \psi_{1} = \mathbf 1\Big\{|T_1| > \uc \sqrt{\frac{\sum_{i>A} p_i}{n}} \Big\}.
\end{equation}

\hfill

\noindent We prove the following proposition regarding this statistic.\\

\begin{proposition}\label{Borne_sup_tail}
There exists $\uc' >0$, such that the following holds. 
\begin{itemize}
    \item Type I error is bounded:
    $$\mathbb P_p(\psi_{1} \lor \psi_2 = 1) \leq \eta/2.$$
    \item Type II error is bounded: for any $q$ such that
$$\norm{q_{>A}}_t \geq \uc'\left( \frac{\norm{p_{>A}}_1^{\frac{2-t}{t}}}{n^\frac{2t-2}{t}}  + \frac{1}{n}\right),$$
it holds that
    $$\mathbb P_q(\psi_{1} \lor \psi_2 = 0) \leq \eta/2.$$
\end{itemize}
\end{proposition}


Recall that the tail is defined such that, \textit{whp} under $\mathcal{H}_0$, no same coordinate is observed at least twice. 
We therefore combine two tests: The test $\psi_2$ rejects $\mathcal{H}_0$ if one of the coordinates is observed at least twice, while the test $\psi_1$ rejects $\mathcal{H}_0$ if the total mass of observed coordinates differs substantially from its expectation under the null. Proposition \ref{Borne_sup_tail} proves that this combination of tests reaches the optimal rate.\\

In \cite{valiant2017automatic}, the tail test only involves the first order moment, which is sufficient in the case of the $\ell_1$ norm. Moreover, in the proof of Proposition \ref{Borne_sup_tail}, it becomes clear that for $t=1$ we only need the test $\psi_1$ and for $t=2$ we only need the test $\psi_2$. However in the case of the $\ell_t$ for $t \in (1,2)$, the combination of both $\psi_1$ and $\psi_2$ is necessary.
\hfill

\subsubsection{Aggregated test}

We now combine the above results to define the aggregated test. We define our test as
$$\psi = \psi_{\bulk} \lor \psi_{1}  \lor \psi_{2} .$$

This is the test rejecting the null whenever one of the three tests does. Denote by 
$$ \bar \rho = \sqrt{\frac{\norm{p_{\leq I}}_r}{n}}+ \frac{\norm{p_{>I}}_1^{\frac{2-t}{t}}}{n^\frac{2t-2}{t}}  + \frac{1}{n}.$$

The following theorem states that this test reaches the rate $\bar \rho$, which is the minimax rate $\rho^*$ given in Theorem \ref{Rate}. 
In other words, it guarantees that, whenever the two hypotheses are $\bar \rho$-separated in $\ell_t$ distance, this test has type-I and type-II errors upper bounded by $\eta/2$, ensuring that its risk is less than $\eta$. Since the minimax separation radius $\rho^*$ is the smallest radius ensuring the existence of a test satisfying this condition, we can conclude that $\rho^* \lesssim \bar \rho$. 

\begin{theorem}
There exists $\uc'>0$, such that the following holds. 
\begin{itemize}
    \item The type I error is bounded:
    $$\mathbb P_p(\psi = 1) \leq \eta/2.$$
    \item The type II error is bounded: for any $q$ such that
$$\norm{p-q}_t \geq \uc'\left( \sqrt{\frac{\norm{p_{\leq I}}_r}{n}}+ \frac{\norm{p_{>I}}_1^{\frac{2-t}{t}}}{n^\frac{2t-2}{t}}  + \frac{1}{n}\right),$$
it holds that
    $$\mathbb P_q(\psi = 0) \leq \eta/2.$$
\end{itemize}
\end{theorem}

\subsection{Remarks on the tests}

In the bulk tests, we propose test statistics based on sample splitting, whose variance is easier to express. However, those tests could be defined slightly differently without sample splitting, allowing also for the analysis of the case $n=1$. Denoting by $H$ the histogram of the data, we could define
$$ \widetilde T_{Bulk} = \sum_{j\leq A} \frac{1}{p_j^b} \left[\Big(\frac{H_j}{n} - p_j\Big)^2 - H_j\right]$$
and the associated test:
$$\widetilde \psi_{\bulk} = \mathbf 1\{\widetilde T_{\bulk} >  \frac{\uc}{n} \norm{p_{\leq A}}_r^\frac{r}{2}\}.$$ This test attains the same upper bound in terms of separation distance - up to multiplicative constants depending on $\eta$ - as the bulk test we define in Equation~\eqref{test_bulk}, and is therefore also optimal in the bulk regime.\\

To understand the interpolation between the extreme cases $t=1$ and $t=2$, an important remark is that the tail tests $\psi_1$ and $\psi_2$ do not capture the same signals. Under the alternative hypothesis, the test $\psi_1$ checks that the total mass of the tail coefficients $\|q_{>A}\|_1$ is not to far away from $\|p_{>A}\|_1$. As to test $\psi_2$, \textit{on the tail}, that is, on a set for which $\sum_{j>A}^N n^2 p_j^2 \ll 1$, it is actually equivalent to using a test for the second order moment. In other words, the test $\psi_2$ is equivalent to $\widetilde \psi_2 = \mathbf 1\{|T_2| >  \frac{\uc}{n} \norm{p_{>A}}_2 \}$ for a small constant $\uc$, where
$$ T_{2} = \sum_{i > A}  \left(\frac{S_i}{k} - p_i \right)\left(\frac{S'_i}{k} - p_i \right).$$ Therefore, the test $\psi_2$ checks that the second order moment of the tail of distribution $q_{>A}$ is not too different from that of $p_{>A}$, in other words, that it does not contain much greater coefficients than the corresponding values of $p_{>A}$.

\section{Further remarks on the results}\label{sec:further_remarks}

\subsection{Influence of the $\ell_t$ norm}\label{Influence_of_the_norm}

In this paper, we consider the separation distance in all $\ell_t$ norms for $t \in [1,2]$. The choice of $t$ influences the minimax separation distance.

In the extreme case $t=2$, the minimax separation distance reduces to: $\rho^* \asymp_\eta \sqrt{\frac{\norm{p_{\leq I}}_2}{n}} + \frac{1}{n}$, which can be further simplified as:
$$\rho^* \asymp_\eta \sqrt{\frac{\left\|p\right\|_2}{n}} + \frac{1}{n}.$$
Indeed, by definition of $I$: $\|p_{>I}\|_2 \lesssim_\eta \frac{1}{n}$. \revision{This case has already been solved in~\cite{chan2014optimal}}. In this case, as discussed earlier, a simple $\chi^2$ test would suffice for reaching this separation distance, and $p$ would only appear in the definition of the threshold of this test. Here we therefore do not need to combine a bulk with a tail test. A single $\chi^2$ test, applied on both the bulk and the tail (i.e. setting $A=N$), would suffice.\\ 


We now consider the opposite extreme case $t=1$. In this case
$$ \rho^* \asymp_\eta \sqrt{\frac{\|p_{\leq I}\|_{2/3}}{n}} + \|p_{>A}\|_1 + \frac{1}{n}.$$
In the minimax separation distance, the contribution of the Bulk coefficients involves the $\ell_{2/3}$ quasi-norm - as in~\cite{valiant2017automatic}. 
In terms of test statistic, this is reflected by the fact that the optimal Bulk test is based on a re-weighted $\chi^2$ test statistic whose weights depend on $p$. For each entry $j$, the optimal weight is larger when $p_j$ is small: indeed, for small $p_j$, coordinate $j$ has smaller variance. This re-weighting differs from the extreme case $t=2$, since, compared to the $\ell_2$ norm, the $\ell_1$ norm lays more emphasis on smaller entries of the perturbation $p-q$. As to the tail coefficients, however, the big picture is simpler as the minimax rate with respect to the tail coefficients is $\|p_{>A}\|_1$, which is very large. This rate implies in particular that only the total mass of the perturbations of the tail coefficients matters. We therefore do not need to use the test $\psi_2$, which is tailored to detect extreme values of the perturbations, and can only restrict to using $\psi_1$ when it comes to the tail coefficients.\\



Between the two extreme cases, that is, for $t \in (1,2)$, we have an interpolation between the two extreme scenarios. When it comes to the bulk, we need to re-weight the test statistics by weights that increase with $p_i$ for entry $i$ as in the case $t=1$. 
But the larger $t$, the milder the reweighting - as the $\ell_t$ norm puts more weight on large coefficients - until it vanishes for $t=2$. 
As for the tail, both tests $\psi_1$ and $\psi_2$ are required in this intermediate regime. Indeed, we need to control both the mass of the tail perturbations like for $t=1$, but also their extreme values like for $t=2$. \revision{Note that~\cite{waggoner2015lp} had already considered the global problem of $\ell_t$ testing for discrete distributions and identified (non-matching) upper and lower bounds.} \\


For $t>2$, the underlying phenomenon is fundamentally different. In this case, the $\ell_t$ norm emphasizes so much the large deviations that re-weighted $\chi^2$ tests - that are related to re-weighted second order moment estimation - seem to be sub-optimal for testing. We leave the case $t>2$ as an open problem.\\ 

In the minimax separation distance in $\ell_t$ norm, the bulk part $\sqrt{\frac{\|p_{\leq I}\|_r}{n}}$ involves a duality between the norms $\ell_t$ and $\ell_r$ 
for $r = \frac{2t}{4-t}$ - as was also the case for $t=1$ in~\cite{valiant2017automatic}. This phenomenon comes from a combination of Hölder's inequality and information theory. Define $\gamma = (\gamma_1, \dots, \gamma_A) \in [0,1]^A$, and define the random vector  $q = (p_1 + \delta_1 \gamma_1, \cdots, p_A + \delta_A \gamma_A)$ for $\delta_i \overset{iid}{\sim} Rad(\frac{1}{2})$ like in \eqref{Prior_bulk}, except that this time, we \textit{do not} impose that $(\gamma_i)_i$ is defined as in \eqref{gamma}. Introduce
$$ \Gamma := \left\{(\gamma_1, \dots, \gamma_A) \in [0,1]^A: \sum_{i=1}^A \frac{\gamma_i^4}{p_i^2} \leq \frac{\Cgamma}{n^2}; \;\; p_i - \gamma_i \in [0,1], \;\; p_i + \gamma_i \in [0,1] \right\},$$  
where $\Cgamma$ is a small enough constant depending only on $\eta$. Then by Lemma \ref{lem:gammaquatre} in the Appendix, whenever $\gamma \in \Gamma$, the $n$ samples\footnote{Although the proof is written for graph samples, it is argued in Subsection~\ref{ss:mult} that it can be transposed to the multinomial or the Poisson settings.} generated from the random vector $q$ have a probability distribution indistinguishable from the null hypothesis $p$. The largest $\gamma \in \Gamma$, when measured in $\ell_t$, therefore provides a lower bound on the minimax separation radius. It is found by solving:
$\max_{\gamma \in \Gamma} \; \sum_{i=1}^A \gamma_i^t$, which can be done using Hölder's inequality:
\begin{align*}
    \sum_{i=1}^A \gamma_i^t = \sum_{i=1}^A \left(\frac{\gamma_i^4}{p_i^2}\right)^{t/4} \hspace{-3mm}p_i^{t/2} \underset{\text{Hölder}}{\leq} \left(\sum_{i=1}^A \frac{\gamma_i^4}{p_i^2}\right)^{t/4} \left(\sum_{i=1}^A p_i^r\right)^{(4-t)/4} \hspace{-3mm} \leq \left(\frac{\Cgamma}{n^2}\right)^{t/4} \hspace{-3mm}\|p\|_r^{1/2t},
\end{align*}
where we have used Hölder's inequality with $a = \frac{4}{t}$ and $b = \frac{4}{4-t}$. Setting $\gamma^*$ the vector on the frontier of $\Gamma$ reaching the equality case in Hölder's inequality, we obtain for fixed $n$: $\|\gamma^*\|_t \propto \|p\|_r^{1/2}$.

As to the contribution of the tail, we refer the reader to the remarks below Proposition~\ref{LBtail}.

\subsection{Asymptotics as $n\to \infty$}

Consider now $p$ as being a \textit{fixed} multinomial distribution, or a fixed vector of Poisson parameters. Then by the definitions of $A$ and $I$, there exists an integer $n_0$ such that for all $n\geq n_0$, we have $I=A=N$. In words, we eventually no longer need to split the distribution into bulk and tail and we can define the bulk as the whole set of coefficients. For $n$ large enough ($n\geq n_0$), the local minimax rate therefore rewrites:
$$ \rho^*(p,n) \underset{n \to \infty}{\asymp} \begin{cases} \sqrt{ \frac{\left\|p^{-\max}\right\|_r}{n}} + \frac{1}{n} & \text{ in the multinomial case} \\ ~~~\sqrt{ \frac{\|p\|_r}{n}} ~~~ + \frac{1}{n}  &\text{ in the binomial or Poisson case}.\end{cases}$$

On the other hand the fast rate $\frac{1}{n}$ asymptotically dominates if $p$ is close to a Dirac multinomial distribution in the multinomial setting, or if e.g.~$p=0$ in the binomial and Poisson setting.\\


\paragraph{Acknowledgments.} Both authors acknowledge fruitful discussions with Alexandre Tsybakov, Cristina Butucea and Rajarshi Mukherjee. The work of A. Carpentier is partially supported by the Deutsche Forschungsgemeinschaft (DFG) Emmy Noether grant MuSyAD (CA 1488/1-1), by the DFG - 314838170, GRK 2297 MathCoRe, by the FG DFG, by the DFG CRC 1294 'Data Assimilation', Project A03, by the Forschungsgruppe FOR 5381 "Mathematical Statistics in the Information Age - Statistical Efficiency and Computational Tractability", Project TP 02, by the Agence Nationale de la Recherche (ANR) and the DFG on the French-German PRCI ANR ASCAI CA 1488/4-1 "Aktive und Batch-Segmentierung, Clustering und Seriation: Grundlagen der KI" and by the UFA-DFH through the French-German Doktorandenkolleg CDFA 01-18 and by the SFI Sachsen-Anhalt for the project RE-BCI.

\newpage

\noindent {\Huge \textbf{APPENDIX}}
\appendix

\section{Lower bound}

Let $p \in \mathcal{P}_N$. For $\mathcal P_1:=\mathcal P_1(\rho)$ a particular collection of elements of $\mathcal{P}_N$ satisfying $\mathcal{H}_{1,\rho}$ we denote by $\mathcal U(\mathcal P_1)$ the uniform distribution over $\mathcal P_1$. 

Let $\mathcal{G} = \left(\{0,1\}^{N}\right)^n$ be the set of all possible observations $(X_1,\dots, X_n)$ where $X_i = (X_i(1), \dots, X_i(N))$. The following lemma gives a way to derive a lower bound on $\rho^*$ by giving a sufficient condition, for a fixed $\rho$, that $R^*(\rho) \geq \eta$:
\begin{lemma}\label{lowerbound} If
$$\frac{1}{|\mathcal{G}|}\sum_{\mathbf X \in \mathcal{G}} \frac{\left(\mathbb{E}_{q \sim \mathcal{U}(\mathcal P_1)} \mathbb{P}_q(X) \right)^2}{\mathbb{P}_p(X)} \leq 1 + 4(1-\eta)^2,$$
Then $R^*(\rho) \geq \eta$.
\end{lemma}

\begin{proof}[Proof of Lemma~\ref{lowerbound}]

We have that:
\begin{align*}
    R^*(\rho) & \geq \inf_{\psi \;\text{test}} \mathbb{P}_p(\psi = 1) + \sup_{q \in \mathcal P_1} \mathbb{P}_q (\psi = 0) \;\;\;\; \text{ (all elements of } \mathcal{P}_1 \text{ satisfy } \mathcal{H}_1\text{)}\\
    & \geq \inf_{\psi \;\text{test}}\mathbb{P}_p(\psi = 1) + \mathbb{E}_{q \sim \mathcal{U}(\mathcal P_1)} \mathbb{P}_q(\psi = 0) \;\;\;\;\; \text{ (the supremum is greater than the integral)}\\
        & = 1 + \inf_{\psi \;\text{test}}\mathbb{P}_p(\psi = 1) - \mathbb{E}_{q \sim \mathcal{U}(\mathcal P_1)} \mathbb{P}_q(\psi = 1)\\
    & = 1 - \sup_{\psi \;\text{test}}\abs{\mathbb{P}_p(\psi = 1) - \mathbb{E}_{q \sim \mathcal{U}(\mathcal P_1)} \mathbb{P}_q(\psi = 1) }\\
    & = 1 - d_{TV}(\mathbb{P}_p,\mathbb{E}_{q \sim \mathcal{U}(\mathcal P_1)} \mathbb{P}_q)\\
    & \geq 1 - \frac{1}{2}\sqrt{\chi^2(\mathbb{E}_{q \sim \mathcal{U}(\mathcal P_1)} \mathbb{P}_q \; ||\;\mathbb{P}_p)},
\end{align*}

where the definition of the $\chi^2$ divergence can be found in \cite{tsybakov2008introduction}, as well as the proof for the inequality $d_{TV} \leq \frac{1}{2} \sqrt{\chi^2}$. Therefore:
\begin{align*}
    R^*(\rho) & \geq 1 - \frac{1}{2}\sqrt{\chi^2(\mathbb{E}_{q \sim \mathcal{U}(\mathcal P_1)} \mathbb{P}_q \; ||\;\mathbb{P}_p)}\\
    & = 1 - \frac{1}{2}\sqrt{\frac{1}{|\mathcal{G}|}\sum_{X \in \mathcal{G}} \frac{\left(\mathbb{E}_{q \sim \mathcal{U}(\mathcal P_1)} \mathbb{P}_q(X) \right)^2}{\mathbb{P}_p(X)} -1}
\end{align*}

Therefore, to have $R^*(\rho) \geq \eta$ it suffices that 
$$\frac{1}{|\mathcal{G}|}\sum_{X \in \mathcal{G}} \frac{\left(\mathbb{E}_{q \sim \mathcal{U}(\mathcal P_1)} \mathbb{P}_q(X) \right)^2}{\mathbb{P}_p(X)} \leq 1 + 4(1-\eta)^2.$$
\end{proof}

For all $i = 1, \cdots, N, $ let $ \gamma_i \in [0,p_i]$ and let $\gamma = \left(\gamma_i \right)_i$. We now apply the previous lemma with 
$$ \mathcal P_1 = \left\{p + \left(\delta_i \gamma_i\right)_{i \leq N} \; |\; \delta \in \{\pm 1\}^{N}  \right\}.$$

\vspace{3mm}

\begin{lemma}\label{lem:gammaquatre}
There exists a sufficiently small absolute constant $\cL{\ref{lem:gammaquatre}}$ such that, if $\sum\limits_{i=1}^{N} \frac{ \gamma^4_i}{p_i^2} \;\leq\; \frac{\cL{\ref{lem:gammaquatre}}}{n^2}$, then for all $\rho \leq \|\gamma\|_t$ we have $R^*(\rho) \geq \eta$.
\end{lemma}

\begin{proof}
We will use Lemma \ref{lowerbound} with $p$ and $\mathcal P_1$ defined as above.

\begin{itemize}
    \item We first compute $\mathbb P_q(X)$ for some realization $X \in \mathcal{G}$. Let $S = \sum_{i=1}^n X_i \in \{0,\dots, n\}^N$ and write $S=(s_1,\dots, s_N)$. We have that
    \begin{align*}
        \mathbb P_p(X) & = \prod_{i=1}^{N} p_i^{s_i}(1-p_i)^{n-s_i}
    \end{align*}
    \item We now compute $\mathbb{E}_{q \sim \mathcal{U}(\mathcal{P}_1)} \mathbb{P}_q(X)$: for any $(\delta_i)_i \in \{\pm1\}^{N}$, we define $q_\delta = p + (\delta_i \gamma_i)_{1 \leq i \leq N}$. Then we have:
    \begin{align*}
        \mathbb P_{q_\delta}(X) = \prod_{i=1}^{N} (p_i + \delta_i \gamma_i)^{s_i}(1-p_i - \delta_i \gamma_i)^{n-s_i}
    \end{align*}
    \end{itemize}
    Therefore we have:
    
    \begin{align*}
            &\frac{1}{|\mathcal{G}|}\sum_{X \in \mathcal{G}} \frac{\left(\mathbb{E}_{q \sim \mathcal{U}(\mathcal{P}_1)} \mathbb{P}_q(X) \right)^2}{\mathbb{P}_p(X)} = \frac{1}{|\mathcal{G}|} \; \sum_{X \in \mathcal{G}} \; \sum_{\delta, \delta'} \; \prod_{i=1}^{N}  \frac{(p_i + \delta_i \gamma_i)^{s_i}(1-p_i - \delta_i \gamma_i)^{n-s_i} }{p_i^{s_i}(1-p_i)^{n-s_i}}\\
        &  \;\;\;\;\;\;\;\;\;\;\;\;\;\;\;\;\;\;\;\;\;\;\;\;\;\;\;\;\;\;\;\;\;\;\;\;\;\;\;\;\;\;\;\;\times (p_i + \delta_i \gamma_i)^{s_i}(1-p_i - \delta_i \gamma_i)^{n-s_i}\\
            &  \\
        & = \frac{1}{|\mathcal{G}|} \sum_{\delta,\delta'} \; \prod_{i=1}^{N}\;\sum_{l = 0}^n \binom{n}{l} \left(p_i + (\delta_i + \delta'_i) \gamma_i + \frac{\delta_i \delta'_i \gamma^2_i}{p_i}\right)^l \left(1 - p_i - (\delta_i + \delta'_i) \gamma_i + \frac{\delta_i \delta'_i \gamma^2_i}{1 - p_i}\right)^{n-l}\\
        &  \\
                & = \frac{1}{|\mathcal{G}|} \sum_{\delta,\delta'} \; \prod_{i=1}^{N} \left(1 + \frac{\delta_i \delta'_i \gamma^2_i}{p_i(1 - p_i)} \right)^n  = \prod_{i=1}^{N} \left[\frac{1}{4} \sum_{\delta_i, \delta'_i \in \{\pm 1\}^n} \left(1 + \frac{\delta_i \delta'_i \gamma^2_i}{p_i(1 - p_i)} \right)^n\right] \\
        &  \\
        & = \prod_{i=1}^{N} \left[ \frac{1}{2}  \left(1 + \frac{ \gamma^2_i}{p_i(1 - p_i)} \right)^n + \frac{1}{2}  \left(1 - \frac{ \gamma^2_i}{p_i(1 - p_i)} \right)^n \right] \\
        &  \\
        & \leq \prod_{i=1}^{N} \left[ \frac{1}{2}  \exp\left(\frac{ n\gamma^2_i}{p_i(1 - p_i)} \right) + \frac{1}{2}  \exp\left(\frac{ -n\gamma^2_i}{p_i(1 - p_i)} \right) \right] \\
        &  \\
        & = \prod_{i=1}^{N} \cosh\left(\frac{ n\gamma^2_i}{p_i(1 - p_i)} \right)  \leq \exp \left( \sum_{i=1}^N \frac{ n^2\gamma^4_i}{2 p_i^2(1 - p_i)^2}\right) 
    \end{align*}
Note that
\begin{align*}
    \exp \left( \sum_{i=1}^N \frac{ n^2\gamma^4_i}{2 p_i^2(1 - p_i)^2}\right) & \leq 1 + 4(1 - \eta)^2\\
    \Longleftrightarrow \;\;\;\;\;\;\;\;\;\; & \sum_{i=1}^N \frac{ \gamma^4_i}{p_i^2(1 - p_i)^2} \leq \frac{2L^2_\eta}{n^2}
\end{align*}

\begin{equation}\label{gamma_quatre}
 \;\;\;\;\;\;\;\;\;\;\Longleftarrow \;\;\;\;\;\;\;\;\;\;\sum_{i=1}^N \frac{ \gamma^4_i}{p_i^2} \;\leq\; \frac{L^2_\eta}{2n^2}\tag{$\arraystretch$}
\end{equation}

where $\cA^4 := \log\left(1 + 4(1-\eta)^2 \right)$ and since $\forall \;i : p_i\leq \frac{1}{2}$. The result follows by Lemma \ref{lowerbound}.

\end{proof}

\hfill

 This means the following: let $\gamma := \left(\gamma_i \right)_i$ satisfying \eqref{gamma_quatre} and let $\rho = \norm{\gamma}_t$. Then all points $p +  (\delta_i \gamma_i)_{1 \leq i \leq |\mathcal G|} $ are located at a distance $\rho$ from $p$ in terms of $\ell_t$ norm - so that the corresponding adjacency matrices are at a distance $\rho$ from each other in $\ell_t$ norm. Moreover we proved that for the uniform prior on this set of points $\mathcal{P}_1$, we have $R^*(\rho) \geq \eta$, which yields $\rho^* \geq \rho$. 
 
\hfill\break


We now prove the lower bound by combining the following four lemmas.\\

\begin{lemma}\label{bulk}
It holds that
$$ \rho_t^* \gtrsim_\eta \rho_1 := \frac{ \norm{p_{\leq A}}_r^\frac{r}{t}}{\sqrt{n}\norm{p_{\leq I}}_r^\frac{r}{4}}.$$
\end{lemma}

\begin{proof}[Proof of Lemma \ref{bulk}]
For a small enough constant $\cA$ depending only on $\eta$, we then define the constant 
\begin{equation}\label{def_a}
    a = \frac{\cA}{\sqrt{n}\left(\sum_{i\leq I} p_i^r \right)^\frac{1}{4}}
\end{equation}
For all $\delta \in \{\pm 1\}^A$ let $q_\delta = ((q_{\delta})_i)_{i = 1, \cdots, N}$ such that 
\begin{itemize}
    \item $\forall i \leq A \;,~~~~(q_{\delta})_i = p_i + a \delta_i p_i^\frac{2}{4-t}$ where $a$ is defined in \eqref{def_a}
    \item $\forall i>A \;,~~~ (q_{\delta})_i = p_i$.
\end{itemize}

Let $\mathcal{P}_1 = \left\{ q_\delta \; | \; \delta \in \{\pm 1\}^A \right\}$. We set a uniform prior on $\mathcal{P}_1$. With the notation of Lemma \ref{lem:gammaquatre}, we just set $\gamma_i = a p_i^\frac{2}{4-t}$ if $i \leq A$ and $0$ otherwise. In terms of $\norm{\cdot}_t$ norm, any probability matrix where this prior puts mass is separated from $p$ with a distance $\rho$ such that:
\begin{align*}
    \rho & = a \left\|\left(p_i^\frac{2}{4-t} \right)_{i=1, \cdots, A}\right\|_t = \frac{\cA}{\sqrt{n}\left(\sum_{i\leq I} p_i^r \right)^\frac{1}{4}} \left(\sum_{i\leq A} p_i^r\right)^\frac{1}{t} \asymp_\eta \frac{\norm{p_{\leq A}}_r^\frac{r}{t}}{\sqrt{n}\norm{p_{\leq I}}_r^\frac{r}{4}} = \rho_1.
\end{align*}

According to Lemma \ref{lem:gammaquatre}, taking $\cA^4 \leq \cL{\ref{lem:gammaquatre}}$ this prior gives a minimax risk greater than $\eta$ since 
 \begin{align*}
     \sum_{i\leq A} \frac{ \gamma^4_i}{p_i^2} & = a^4 \sum_{i \leq A} p_i^{\frac{8}{4-t} - 2} = \frac{\cA^4}{n^2} \leq \frac{\cL{\ref{lem:gammaquatre}}}{n^2}.
\end{align*}
\end{proof}

\begin{lemma}\label{tail}
Assume that $\norm{p_{> I}}_1 \geq \frac{1}{n}$. Then it holds that
$$ \rho_t^* \gtrsim_\eta \rho_2:=\frac{\norm{p_{\geq I}}_1^{\frac{2-t}{t}}}{n^\frac{2t-2}{t}}.$$
\end{lemma}

\begin{proof}[Proof of Lemma \ref{tail}]
We divide the proof in two steps. In the first step, we prove that the prior concentrates with high probability on a zone located at $\frac{\|p_{\geq U}\|_1^{(2-t)/t}}{n^{(2t-2)/t}} + \frac{1}{n}$, up to a multiplicative constant. In the second step, we prove that the prior is indistinguishable from the null hypothesis $p$, by proving that the total variation between $p$ and this prior is small.\\

\textbf{\underline{FIRST STEP}}: We prove that the prior concentrates with high probability on a zone located at $\frac{\|p_{\geq U}\|_1^{(2-t)/t}}{n^{(2t-2)/t}} + \frac{1}{n}$, up to a multiplicative constant. By assumption we have $\norm{p_{> I}}_1 \geq \frac{1}{n}$.

Let $U$ be the smallest index greater than or equal to $A$ such that $n^2 p_U \norm{p_{\geq U}}_1 \leq c_u$ where $c_u= \frac{\eta}{10} \wedge \frac{1}{2}(1-\eta)^2$ \\

Let 
$$ \bar \pi = \frac{c_u}{ n^2\norm{p_{\geq U}}_1} \text{ and } \pi_i = \frac{p_i}{\bar \pi}.$$\\

We set the following sparse prior on the matrices of connection probability: for all $i<U$ we set $q_i = p_i$ and for all $ i \geq U$ we draw $b_i \sim \mathcal B (\pi_i)$ mutually independent, and we define $q_i = b_i \bar \pi$. We write $q = (q_i)_i$ and $q$ for the corresponding random connection probability matrix - for which we write $\mathcal Q$ for the distribution.\\

Before showing that the data distribution coming from this prior - namely $\mathbb E_{q\sim \mathcal Q} \mathbb P_q$ - is close enough to $\mathbb P_\pi$ in total variation, we first prove that $q\sim \mathcal Q$ is such that $\|q - p\|_t$ is with high probability larger - up to a positive multiplicative constant that depends only on $u$ - than $\rho_2$ from the null. 
We have 
\begin{align*}
    \mathbb{E}_{q \sim \mathcal Q}\left[ \norm{p - q}_t^t  \right]&= \mathbb{E}_{(b_i)_i \sim   \otimes \mathcal B(\pi_i)}\left[\sum_{i \geq U} \abs{b_i \bar \pi - p_i}^t \right]\\
    &= \bar \pi^t \mathbb{E}_{(b_i)_i \sim  \otimes \mathcal B(\pi_i)}\left[\sum_{i\geq U} \abs{b_i - \pi_i}^t \right]\\ 
    &= \bar \pi^t \sum_{i\geq U} \pi_i(1 - \pi_i)^t + (1- \pi_i)\pi_i^t \geq  4^{-1}\bar \pi^t \sum_{i \geq U} \pi_i + \pi_i^t \;  \\
    &\geq  \;  4^{-1}\bar \pi^t \sum_{i \geq U} \pi_i,
\end{align*} 
since $\forall i\geq U, \; \pi_i \leq c_u \leq \frac{1}{2},$ and 
\begin{align*}
    \mathbb{V}_{q \sim \mathcal Q}\left[\; \norm{p - q}_t^t \; \right]& = \bar \pi^{2t}\sum_{i\geq U}\mathbb{V}_{b_i \sim \mathcal B(\pi_i)}\abs{b_i - \pi_i}^t = \bar \pi^{2t}\sum_{i \geq U} \pi_i(1-\pi_i)\left[(1-\pi_i)^t - \pi_i^t\right]^2 \\ 
&\leq \bar \pi^{2t} \sum_{i \geq U} \pi_i.
\end{align*}

We now show that $\left[\mathbb{E}_{q \sim \mathcal Q}\left[ \norm{p - q}_t^t  \right]\right]^2 \gg \mathbb{V}_{q \sim \mathcal Q}\left[ \norm{p - q}_t^t  \right]$. This is equivalent to proving $\sum_{i \geq U} \pi_i \gg 1$, or equivalently: $n^2 \norm{p_{\geq U}}_1^2 \gg c_u$.\\

By Lemma \ref{I_and_U}, we are necessarily in the case $\norm{p_{\geq U}}_1 \geq \frac{1}{3} \norm{p_{>I}}_1$. Indeed, suppose that $\norm{p_{\geq U}}_1 < \frac{1}{3} \norm{p_{>I}}_1$, then by Lemma \ref{I_and_U} we would have 
\begin{align*}
    \norm{p_{>I}}_1 &\leq  \norm{p_{\geq U}}_1 + \frac{\sqrt{\cI}}{n}\\ &\leq \frac{1}{3} \norm{p_{>I}}_1 + \frac{\sqrt{\cI}}{n},
\end{align*} hence $\norm{p_{>I}}_1 \leq   \frac{3}{2}\frac{\sqrt{\cI}}{n}$, which is excluded because we assume $ \norm{p_{>I}}_1 \geq \frac{1}{n}$.
    
Therefore, $\norm{p_{\geq U}}_1^2 n^2 \geq \frac{1}{9} \gg c_u$. We conclude using Chebyshev's inequality. Therefore, this prior is indeed separated away from the null distribution by a distance greater than $\bar \pi \sum_{i \geq U} \pi_i$ up to a constant, or equivalently, greater than $\frac{\norm{p_{\geq U}}_1^{\frac{2-t}{t}}}{n^{\frac{2(t-1)}{t}}}$.\\

\textbf{\underline{SECOND STEP}}: We now show that this prior is indistinguishable from $p$, i.e. that that is has a bayesian risk strictly greater than $eta$. 
We write $\bar{\mathbb P}_{\mathrm{tail}} = \mathbb{E}_{q \sim \mathcal Q}\left[\mathbb{P}_q \right]$, the prior distribution used to lower bound the minimax risk. We always have:

\begin{align*}
    R^* & \geq 1 - d_{TV}\left(\mathbb P_p, \bar{\mathbb P}_{\mathrm{tail}} \right).
\end{align*}

Moreover, we recall that for any realization $X = (X_1, \dots, X_n)$ we write $S = \sum_{i=1}^n X_i$. We have
\begin{align*}
    & d_{TV}\left(\mathbb P_p, \bar{\mathbb P}_{\mathrm{tail}} \right)= \frac{1}{2}\sum_{X \in \mathcal G}\abs{\mathbb P_p(X) - \bar{\mathbb P}_{\mathrm{tail}}(X)}\\
    & = \frac{1}{2}\sum_{X \in \mathcal G: \forall i \geq U, s_i\leq 1}\abs{\mathbb P_p(X) - \bar{\mathbb P}_{\mathrm{tail}}(X)} \;\; + \;\; \frac{1}{2}\sum_{X \in \mathcal G: \exists i \geq U,~\mathrm{s.t.}~s_i \geq 2}\abs{\mathbb P_p(X) - \bar{\mathbb P}_{\mathrm{tail}}(X)}.
\end{align*}

This allows us to split the total variation into two terms: The first one will be the principal term, while the second one will be negligible. We first prove the negligibility of the second term.

We have - since $s$ is a sufficient statistic
\begin{align*}
    & \sum_{X \in \mathcal G: \exists i \geq U,~\mathrm{s.t.}~s_i \geq 2}\abs{\mathbb P_p(X) - \bar{\mathbb P}_{\mathrm{tail}}(X)} \leq  \left[ \mathbb P_p\left(\exists i \geq U \; ; \; s_i \geq 2 \right) + \bar{\mathbb P}_{\mathrm{tail}}\left(\exists i \geq U \; ; \; s_i \geq 2 \right) \right]\\
    & \leq \sum_{i = U}^{|\mathcal G|} \left[1 - \mathbb P_p( s_i = 0)- \mathbb P_p(s_i = 1) + 1 - \bar{\mathbb P}_{\mathrm{tail}}( s_i = 0)- \bar{\mathbb P}_{\mathrm{tail}}(s_i = 1)\right].
\end{align*}

Let's fix $i \in \{U, \cdots, N\}$. We will use the following inequalities which hold for all $ n \in \mathbb{N}, x \in [0,1]$:

\begin{align*}
    (1-x)^n \geq 1-nx ; \;\;\;\;\; (1-x)^n \geq 1-nx + \frac{n}{4}x^2 ; \;\;\;\;\; (1-x)^n \leq 1-nx + \frac{n^2}{2}x^2.
\end{align*}

\paragraph{First term in the sum: $\sum_{i = U}^{N} [1 - \mathbb P_p( s_i = 0)- \mathbb P_p(s_i = 1)]$.} We recall that by the definition of $U$ we have $\forall i \geq U \; n p_i \leq c_u$ so that for any $i \geq U$
\begin{align*}
    &   1 - \mathbb P_p( s_i = 0)- \mathbb P_p(s_i = 1)  \;\; = \;\; 1 - (1 - p_i)^n - n p_i (1-p_i)^{n-1}\\
    &\\
    & \leq  1 - \left[1 - np_i + \frac{n}{4}p_i^2 \right] - np_i\left[ 1 - (n-1)p_i \right] \;\; \leq \;\; n^2 p_i^2.
\end{align*}
Summing over all $i = U, \cdots,N$ yields that 
$$ \sum_{i = U}^{N}  [1 - \mathbb P_p( s_i = 0)- \mathbb P_p(s_i = 1)]  \leq  \cI.$$

\paragraph{Second term in the sum: $\sum_{i = U}^{N} [1 - \bar{\mathbb P}_{\mathrm{tail}}( s_i = 0)- \bar{\mathbb P}_{\mathrm{tail}}(s_i = 1)]$.} We recall that by the definition of $U$ we have $\forall i \geq U \; n p_i \leq c_u$ so that for any $i \geq U$
\begin{align*}
    & 1 - \bar{\mathbb P}_{\mathrm{tail}}(s_i = 0) - \bar{\mathbb P}_{\mathrm{tail}}(s_i = 1)\;\; = \;\; 1 - \left[ 1-\pi_i + \pi_i(1 - \bar \pi)^n \right] -  \pi_i n \bar \pi (1-\bar \pi)^{n-1}\\
    & = \pi_i - \pi_i(1-\bar \pi)^n - \pi_i n \bar \pi (1-\bar \pi)^{n-1} \leq \pi_i - \pi_i(1-n\bar \pi) - \pi_i n \bar \pi (1-(n-1)\bar \pi)\\
    & = n(n--1) \pi_i \bar \pi^2 = n(n-1) p_i \bar \pi \leq n^2 c_u \frac{p_i}{n^2\|p_{\geq U}\|_1} = c_u \frac{p_i}{\|p_{\geq U}\|_1}
\end{align*}

Summing over all $i = U, \cdots, N$ yields that
$$ \sum_{i = U}^{N} [1 - \bar{\mathbb P}_{\mathrm{tail}}(s_i = 0) - \bar{\mathbb P}_{\mathrm{tail}}(s_i = 1)]  \leq c_u \frac{\norm{p_{\geq U}}_1}{\norm{p_{\geq U}}_1} = c_u.$$

Therefore 
\begin{equation}\label{TV}
    d_{TV}\left(\mathbb P_p, \bar{\mathbb P}_{\mathrm{tail}} \right) = \underbrace{\frac{1}{2}\sum_{X \in \mathcal G: \forall i \geq U, s_i\leq 2}\abs{\mathbb P_p(X) - \bar{\mathbb P}_{\mathrm{tail}}(X)} }_{\text{principal term}} + \cI + c_u
\end{equation}

\hfill \break

Now, we can upper bound the total variation by the $\chi^2$ divergence on the high probability event that we only observe $0$ or $1$ for each coordinate $i \geq U$ corresponding to the principal term. We have - since $s$ is a sufficient statistic
\begin{align}
    & \sum_{X \in \mathcal G: \forall i \geq U, s_i\leq 1}\abs{\mathbb P_p(X) - \bar{\mathbb P}_{\mathrm{tail}}(X)} \\
    \leq &\sqrt{\sum_{X \in \mathcal G: \forall i \geq U, s_i\leq 1}\frac{\left(\mathbb P_p(X) - \bar{\mathbb P}_{\mathrm{tail}}(X)\right)^2}{\mathbb P_p(X)}}\sqrt{\underbrace{\sum_{X \in \mathcal G: \forall i \geq U, s_i\leq 1}\mathbb P_p(X)}_{\leq 1}} \nonumber\\
    & \leq \sqrt{\sum_{X \in \mathcal G: \forall i \geq U, s_i\leq 1} \frac{ \bar{\mathbb P}_{\mathrm{tail}}(X)^2}{\mathbb P_p(X) } - 1 + 2c_u} = \sqrt{\prod_{i=U}^N \left(\sum_{j=0}^1 \frac{ \bar{\mathbb P}_{\mathrm{tail}}(s_i = j)^2}{\mathbb P_p(s_i = j) } \right) - 1 + 2c_u}.\label{approx_chi2}
\end{align}

\paragraph{Computation of $\sum_{k=0}^1 \frac{ \bar{\mathbb P}_{\mathrm{tail}}(s_i = k)^2}{\mathbb P(s_i = k) }$.}:

\begin{align*}
    & \sum_{k=0}^1 \frac{ \bar{\mathbb P}_{\mathrm{tail}}(s_i = k)^2}{\mathbb P_p(s_i = k) }  =  \;\; \frac{\left[1 - \pi_i + \pi_i(1-\bar \pi)^n\right]^2}{(1-p_i)^n} \;\; + \;\; \frac{\left[\pi_i n \bar \pi(1- \bar \pi)^{n-1} \right]^2}{np_i(1-p_i)^{n-1}}
\end{align*}

The first term writes:

\begin{align*}
    & \frac{\left[1 - \pi_i + \pi_i(1-\bar \pi)^n\right]^2}{(1-p_i)^n} \;\; \leq \;\;  \frac{\left[1 - \pi_i + \pi_i(1-n \bar \pi + \frac{n^2}{2} \bar \pi^2)\right]^2}{1-np_i}\\
    &  = \;\; 1 - np_i + n^2p_i \bar \pi + \frac{\left(\frac{n^2}{2}p_i \bar \pi \right)^2}{1-np_i} \leq 1 - np_i + n^2p_i \bar \pi + \frac{n^4p_i^2 \bar \pi^2}{4(1-\cI)} \;\; \\
    & \leq \;\; 1 - np_i + n^2p_i \bar \pi + \frac{c_u^2}{4(1-\cI)}.
\end{align*}

The second term writes:

\begin{align*}
    & \frac{\left[\pi_i n \bar \pi(1-\bar \pi)^{n-1} \right]^2}{np_i(1-p_i)^{n-1}} \;\;\; = \;\;\; np_i \frac{(1-\bar \pi)^{2n-2}}{(1-p_i)^{n-1}} \;\;\; \leq \;\;\; n p_i \;\;\; \text{ since } \bar \pi \geq p_i
\end{align*}

\hfill

We can now sum the two terms:
$$ \sum_{k=0}^1 \frac{ \bar{\mathbb P}_{\mathrm{tail}}(s_i = k)^2}{\mathbb P_p(s_i = k) }  = 1 + n^2p_i \bar \pi + \frac{c_u^2}{4(1-\cI)}$$
So that
\begin{align*}
    &\prod_{i=U}^N \left(\sum_{k=0}^1 \frac{ \bar{\mathbb P}_{\mathrm{tail}}(s_i = k)^2}{\mathbb P_p(s_i = k) } \right) =  \prod_{k=U}^N \left(1 + n^2p_i  \bar \pi + \frac{c_u^2}{4(1-\cI)} \right) \\
    & = \exp{\left(c_u + \frac{c_u^2}{1-\cI} \right)} \leq \exp{\frac{3}{2}c_u} \leq 1 + 3c_u \;\;\; \text{ since } \frac{3}{2}c_u \leq 1.
\end{align*}

Now, using (\ref{TV}) and (\ref{approx_chi2}), we have: $d_{TV}(\mathbb P_p,\bar{\mathbb P}_{\mathrm{tail}}) \leq \frac{1}{2}\sqrt{5c_u} + \cI + c_u \leq 1 - \eta$ by the definition of $c_u, \cI$. This concludes the proof.

\end{proof}

\begin{lemma}\label{1_over_n}
Assume that $\norm{p_{\geq I}}_1 \leq \frac{1}{n}$. Then it holds that
$$ \rho_t^* \gtrsim \rho_3 := \frac{1}{n}.$$
\end{lemma}

\begin{proof}[Proof of Lemma \ref{1_over_n}]
We introduce $q$ such that $q_1 = p_1 + \frac{1-\eta}{n}$ and $q_j = p_j$ for all $j \geq 2$.

\begin{align*}
    R^* &\geq \inf_{\psi \textbf{test}}\mathbb P_p(\psi = 1) + \mathbb P_{q}(\psi = 0) \;\; = \;\; 1 - d_{TV}(\mathbb P_p, \mathbb P_{q})\\
    &= \;\; 1 - n \;d_{TV}\left(\underset{i<j}{\bigotimes} \mathcal{B}(p_i),\underset{i<j}{\bigotimes} \mathcal{B}(q_i)\right)\\
    & = 1 - n\; d_{TV}\left(\,\mathcal{B}(p_1),\; \mathcal{B}(q_1)\,\right) = 1 - n\; \abs{p_1 - q_1} = 1 - n\; \frac{1-\eta}{n}\\
    & = \eta.
\end{align*}
This concludes the proof.
\end{proof}

\hfill

\begin{lemma}\label{I_and_U} It holds : $\norm{p_{\geq U}}_1 + \frac{1}{n} \asymp \norm{p_{>I}}_1 + \frac{1}{n}$.\\ Moreover, we have either $\norm{p_{\geq U}}_1 \geq \frac{1}{3}\norm{p_{>I}}_1$ or $\norm{p_{>I}}_1 \leq \norm{p_{\geq U}}_1 + \frac{\sqrt{\cI}}{n}$
\end{lemma}

\begin{proof}[Proof of lemma \ref{I_and_U}]

If $\norm{p_{\geq U}}_1 \geq \frac{1}{3}\norm{p_{>I}}_1$ then the result is clear. Now, suppose $\norm{p_{\geq U}}_1 < \frac{1}{3}\norm{p_{>I}}_1$. We have $\norm{p_{\geq U}}_1 < \frac{1}{2} \norm{P_{I \rightarrow U}}$ where $P_{I \rightarrow U} = (p_{I+1}, \cdots, p_{U-1})$. We have:

\begin{align*}
    p_{U-1}^2 + \frac{\cI}{2n^2} & \;\;\geq \;\;  p_{U-1}^2 + \frac{1}{2} \sum_{i=I+1}^{U-1}p_i^2 \;\; \geq \;\; p_{U-1}\left(p_{U-1} + \frac{1}{2} \sum_{i=I+1}^{U-1}p_i \right) \;\;\\
    &  > \;\; p_{U-1}\left(p_{U-1} + \sum_{i \geq U}p_i \right) \\
    & \geq  p_{U-1} \sum_{i \geq U-1} p_i \;\; = \;\; p_{U-1} \norm{P_{\geq U-1}}_1 > \frac{c_u}{n^2}
\end{align*}
by the definition of $U$. \\

Therefore, $$p_{U-1}^2 > \frac{2c_u-\cI  }{2n^2} \Longrightarrow \forall I < i < U, \;\; p_i^2 > \frac{\cI}{2n^2} \;\; \text{ since } c_u \geq \cI.$$

Moreover, 

$$ \frac{\cI}{n^2} \geq \sum_{I< i <U} p_i^2 > (I-U-1)p_{U-1}^2 > (I-U-1)\frac{\cI    }{2n^2}$$

So that 

$$ I-U-1 < 2 \;\;\; i.e. \;\;\; I-U-1 \leq 1$$

Thus: 
\begin{align*}
    \norm{p_{>I}}_1 &\leq \norm{P_{I \rightarrow U}}_1 + \norm{p_{\geq U}}_1 \leq (I-U-1)p_{I+1} + \norm{p_{\geq U}}_1\\
    &\leq \frac{\sqrt{\cI}}{n} + \norm{p_{\geq U}}_1 \lesssim \norm{p_{\geq U}}_1  + \frac{1}{n}.
\end{align*}

Hence the result.
\end{proof} 

\begin{lemma}\label{A_and_I}
Let $\rho_1$ and $\rho_2$ be defined as in Lemmas~\ref{bulk} and \ref{tail}. We have 
$\rho_1 + \rho_2 \asymp \sqrt{\frac{\norm{p_{\leq I}}_r}{n}} + \rho_2$.
\end{lemma}

\begin{proof}[Proof of Lemma \ref{A_and_I}]
Clearly, $\rho_1 + \rho_2 \leq \sqrt{\frac{\norm{p_{\leq I}}_r}{n}} + \rho_2$. To prove $\rho_1 + \rho_2 \gtrsim_\eta \sqrt{\frac{\norm{p_{\leq I}}_r}{n}} + \rho_2$, there are two cases.
\begin{itemize}
    \item If $A=I$ then the result is clear.
    \item Otherwise, $I>A$. Note that by setting $p_i' := np_i$ for all $i =1, \cdots, N$, the result to show can be rewritten as: 
    \begin{equation}\label{P_prime}
        \frac{\norm{p'_{\leq A}}_r^\frac{r}{t}}{\norm{p'_{\leq I}}_r^\frac{r}{4}} + \norm{p'_{\geq I}}_1^{2-t} \asymp \sqrt{\norm{p'_{\leq I}}_r} + \norm{p'_{\geq I}}_1^{2-t}.
    \end{equation}
    We have by definition of $A$ and $I$:
    \begin{align*}
        &p_I'^{2-r}\left( \sum_{i\geq I} p_i' \right)^{2-r} = \left( \sum_{i\geq I}p_I' p_i' \right)^{2-r} \geq \left(\sum_{i\geq I}p_i'^2 \right)^{2-r} \gtrsim_\eta 1 \text{ and}\\
        &\\
        &p_I'^{\;2b} \sum_{i \leq I} p_i'^r \; \leq \; p_{A+1}'^{\;2b} \sum_{i \leq I} p_i'^r \leq \cA^4 \asymp 1 \text{ by definition of } A.
    \end{align*}
    Hence, by noticing that $2b = 2-r$ we have $ \left( \sum_{i\geq I} p_i' \right)^{2-r} > \sum_{i \leq I} p_i'^r$, which yields $ \norm{p'_{\geq I}}_1^{2-t} \geq \sqrt{\norm{p'_{\leq I}}_r} \geq \frac{\norm{p'_{\leq A}}_r^\frac{r}{t}}{\norm{p'_{\leq I}}_r^\frac{r}{4}}  $ by raising to the power $\frac{1}{2r}$. This condition yields the result of the lemma, by replacing $p'$ by $np$.
\end{itemize}
\end{proof}

\begin{lemma}\label{I_to_A}
$\norm{p_{>I}}_1 + \frac{1}{n} \asymp \norm{P_{>A}}_1 + \frac{1}{n}$.
\end{lemma}

\begin{proof}[Proof of lemma \ref{I_to_A}]
If $A = I$ then the result is clear. Now, suppose that $A < I$. We have, by the definition of $A$:
\begin{align*}
    &\frac{\cA^4}{n^2} > P_{A+1}^{2b} \sum_{i \leq I}p_i^r \geq \sum_{i = A+1}^I p_i^2 \geq p_I\sum_{i=A+1}^I p_i ~~ \Longrightarrow ~~ \frac{\cA^4}{n^2 \sum_{i=A+1}^I p_i} \geq p_I
\end{align*}

Moreover if $I<N$, 
\begin{align*}
    \frac{\cI}{n^2} \leq \sum_{i >I} p_i^2 \leq p_{I+1} \sum_{i > I} p_i ~~ \Longrightarrow ~~ p_{I+1} \geq \frac{\cI}{n^2\sum_{i > I} p_i}
\end{align*}

So that 
\begin{align*}
    \sum_{i > I} p_i \geq \frac{\cI}{\cA^4}\sum_{i = A+1}^I p_i
\end{align*} and consequently $\|p_{>I}\|_1 \gtrsim \|p_{>A}\|_1$ if we impose moreover that $\cA^4 \gtrsim \cI$, which can be done \textit{wlog}.

Now if $I=N$, we have $\|p_{>I}\|_1 = 0$ and $p_N > \frac{\sqrt{\cI}}{n}$ and 
\begin{align*}
    p_{A+1}^{2b} < \frac{\cA^4}{n^2\sum_{i=1}^N p_i^r} &\Longrightarrow \sum_{j=A+1}^N p_{A+1}^{2b} p_j^r \leq \frac{\cA^4}{n^2}\\
    &\\
    & \Longrightarrow \sum_{j=A+1}^N p_j^2 \leq \frac{\cA^4}{n^2}\\ 
    & \Longrightarrow P_N \|p_{>A}\|_1 \leq \frac{\cA^4}{n^2}\\ 
    & \Longrightarrow \frac{\sqrt{\cI}}{n} \|p_{>A}\|_1 \leq \frac{\cA^4}{n^2}
\end{align*}

hence $\|p_{>A}\|_1 \lesssim \frac{1}{n} $ so that $\|p_{>A}\|_1 + \frac{1}{n}\asymp \|p_{>I}\|_1 + \frac{1}{n} \asymp \frac{1}{n}$


\end{proof}
\section{Upper bound}

Define $\Delta = q-p$. In the following, $c>0$ denotes an absolute constant, depending only on $\eta$. We call 
$$\rho = \sqrt{\frac{\norm{p_{\leq I}}_r}{n}} + \frac{\norm{p_{\geq I}}_1^{\frac{2-t}{t}}}{n^\frac{2-2t}{t}} + \frac{1}{n},$$
and we prove: $\rho^* \lesssim_\eta \rho$.

We start with the three following lemmas which control the expectation and variance of the statistics $T_{\bulk}, T_1, T_2$. We recall that $\bar n = \lfloor \frac{n}{2} \rfloor$.
\begin{lemma}[Bounds on expectation and variance of $T_{\bulk}$]\label{expVar}
Let $T_{\bulk}$ be defined as in equation~\eqref{test_bulk}. The expectation and variance of $T_{\bulk}$ satisfy:

\begin{align*}
    \mathbb{E}\left[T_{\bulk}\right] & = \sum_{i \leq A} \frac{\Delta_i^2}{p_i^b},\\
    \mathbb{V}[T_{\bulk}] & \leq \sum_{i \leq A} \;\; \frac{1}{p_i^{2b}} \left(\frac{q_i^2}{\bar n^2} + \frac{2}{\bar n}q_i\Delta_i^2\right).
\end{align*}
\end{lemma}

\begin{lemma}[Bounds on expectation and variance of $T_{1}$]\label{expVarT1}
Let $T_1$ be defined as in equation~\eqref{test_psi1}. The expectation and variance of $T_1$ satisfy:

\begin{align*}
    \mathbb{E}\left[T_1\right] & = \sum_{i > A} q_i - p_i,\\
    \mathbb{V}[T_1] & \leq \sum_{i > A} \;\; \frac{q_i}{n}.
\end{align*}
\end{lemma}

\begin{lemma}[Bounds on expectation and variance of $T_{2}$]\label{expVarT2}
Let $T_2$ be defined as in equation~\eqref{test_psi2}. The expectation and variance of $T_2$ satisfy:

\begin{align*}
    \mathbb{E}\left[T_2\right] & = \norm{(p-q)_{>A}}_2^2,\\
    \mathbb{V}[T_2] & \leq \sum_{i > A} \;\; \frac{q_i^2}{\bar n^2} + \frac{2}{\bar n}q_i\Delta_i^2.
\end{align*}
\end{lemma}

We then study the null and alternative hypotheses in the following subsection, bounding the probability of error of the test $\psi$.

\hfill

\subsection{Under the null hypothesis $\mathcal{H}_0$.}

We start by assuming that $p=q$. We recall that $\uc= \frac{4}{\sqrt{\eta}}$.

\paragraph{Test $\psi_{\bulk}$.} Moreover, for the bulk, since $p=q$, we have by lemma \ref{expVar}: $ \mathbb{E}[T_{\bulk}] = 0$ and $\mathbb{V}[T_{\bulk}] = \sum_{i \leq A} \frac{p_i^r}{n^2}$. Therefore by Chebyshev's inequality:
$$\mathbb P\left(T_{\bulk} > \uc \sqrt{\sum_{i \leq A} \frac{p_i^r}{n^2}} \right) \leq \frac{\eta}{16}$$
so that:
\begin{equation}\label{Type_I_error_TBulk}
    P\left(\psi_{\bulk} = 1 \right) \leq \frac{\eta}{16},
\end{equation}

\paragraph{Test $\psi_1$.} Since $p=q$, we have by Lemma~\ref{expVarT1} that $\mathbb{E}(T_{1}) = 0$ and $\mathbb{V}(T_{1}) \leq \sqrt{\frac{\sum_{i>A} p_i}{n}}$. \\
By the same argument $\psi_1$'s type-I error is upper bounded as:
\begin{align*}
    \mathbb P_p\left(\psi_1 = 1 \right) = \mathbb P_p\left(T_1 > \uc\sqrt{\frac{\sum_{i>A} p_i}{n}} \right) \leq \frac{1}{\uc^2} = \frac{\eta}{16},
\end{align*}
so that by definition of $\psi_1$
\begin{equation}\label{Type_I_error_T1}
    \mathbb P_p\left(\psi_1 = 1 \right) \leq \frac{\eta}{16},
\end{equation}

\paragraph{Test $\psi_{2}$.} 
Finally, under the null and since $p=q$, we have $\mathbf{E}(T_2) = 0$ and $\mathbb{V}(T_2) = \frac{1}{n^2} \sum_{i>A} p_i^2$ by Lemma~\ref{expVarT2} so that 

$$ \mathbb{P}\left(T_2 > \uc \sqrt{\frac{\sum_{i>A} p_i^2}{n^2}}\right) \leq \frac{\eta}{16},$$
which rewrites:
\begin{equation}\label{Type_I_error_T2}
    P\left(\psi_2 = 1 \right) \leq \frac{\eta}{16}.
\end{equation}

\paragraph{Conclusion}: Putting together equations \eqref{Type_I_error_T1}, \eqref{Type_I_error_TBulk} and \eqref{Type_I_error_T2} we get that the type I error of $\psi = \psi_{\bulk} \vee \psi_1 \vee \psi_2$ is upper bounded as
$$\mathbb P\left(\psi = 1 \right) \leq \sum_{i \in \{\bulk, 1,2\}} \mathbb P\left(\psi_i = 1 \right) \leq \frac{3\eta}{16} < \eta/2.$$

\subsection{Under the alternative hypothesis $\mathcal H_1( \rho)$} 


Suppose that for some constant $\barc>0$, we have $\norm{\Delta}_t \geq 2 \barc\rho$. By the triangle inequality, there are two cases: 
\begin{itemize}
    \item \underline{\bf First case:} Either $\norm{\Delta_{\leq A}}_t \geq \barc\rho$
    \item \underline{ \bf Second case:} Or $\norm{\Delta_{> A}}_t \geq \barc\rho$
\end{itemize}
 
\hfill

\begin{proposition}[Study in the \underline{\bf First case}]\label{prop:fcase}
There exists a large enough constant $\barc^{(\bulk)}>0$ such that if $\norm{\Delta_{\leq A}}_t \geq \barc^{(\bulk)} \rho$, then
$$\mathbb P(\psi_{\bulk}=1) \geq 1 - \eta/6.$$
\end{proposition}

\begin{proposition}[Study in the \underline{\bf Second case}]\label{prop:scase}
If $\norm{\Delta_{>A}}_t \geq c\rho$, then
$$\mathbb P(\psi_1 \lor \psi_2=1) \geq 1 - \frac{2\eta}{3}.$$
\end{proposition}

\hfill
\hfill
\hfill

\begin{proof}[Proof of Proposition~\ref{prop:fcase}]

Suppose $\norm{\Delta_{\leq A}}_t \geq c\rho$ for some constant $c$. We show that if $c$ is large enough, then the test $\psi_{Bulk}$ will detect it. To do so, we compute a constant $c'$ depending on $c$ such that if $\norm{\Delta_{\leq A}}_t \geq c\rho$, then $\mathbb{V}(T_{Bulk}) \leq c' \; \mathbb{E}(T_{Bulk})^2$ and such that $\lim_{c \to +\infty} c' = 0$. 

By definition of $\rho$, we have in particular: $\norm{\Delta_{\leq A}}_t \geq c\sqrt{\frac{\norm{p_{\leq I}}_r}{n}} \vee \frac{c}{n}$, hence 

\begin{equation}
    \frac{1}{n^2} \leq \frac{1}{c^4}\frac{\norm{\Delta_{\leq A}}_t^4}{\norm{p_{\leq I}}_r^2} \wedge  \frac{\norm{\Delta_{\leq A}}_t^2}{c^2}
\end{equation}
Using Lemma \ref{expVar} we split $\mathbb{V}[T_{\bulk}]$ into four terms

\begin{align*}
    \mathbb{V}[T_{\bulk}] & \leq \sum_{i \leq A} \;\; \frac{1}{p_i^{2b}} \left(\frac{\left(p_i + \Delta_i\right)^2}{n^2} + \frac{2}{n}\left(p_i + \Delta_i\right)\Delta_i ^2\right)\\
    &\\
    & \leq \underbrace{\frac{2}{n^2}\sum_{i \leq A} p_i^r}_{\circled{1}} \;+\; \underbrace{\frac{2}{n^2}\sum_{i \leq A} \frac{\Delta_i^2}{p_i^{2b}}}_{\circled{2}} \;+\; \underbrace{\frac{2}{n}\sum_{i \leq A} p_i^{1 - 2b} \Delta_i^2}_{\circled{3}}  \;+\; \underbrace{\frac{2}{n}\sum_{i \leq A} \frac{\Delta_i^3}{p_i^{2b}}}_{\circled{4}}.
\end{align*}

Now, we show that each of the four terms is less than $\mathbb{E}[T_{\bulk}]^2$, up to a constant\\

{\bf Term $\circled{1}$:} We have by Hölder's inequality:
\begin{align}
    \sum_{i\leq A} \Delta_i^t &\leq \left[ \sum_{i \leq A}  \left(\frac{\Delta_i^t}{p_i^{\frac{b t}{2}}} \right)^\frac{2}{t} \right]^\frac{t}{2}\left[ \sum_{i \leq A} \left(p_i^\frac{b t}{2} \right)^\frac{2}{2-t} \right]^\frac{2-t}{2} = \left(\sum_{i \leq I} \frac{\Delta_i^2}{p_i^b} \right)^\frac{t}{2} \left(\sum_{i \leq I} p_i^r\right)^{1-\frac{t}{2}}. \nonumber \\
    \text{Hence} \;\;\;\;\;\; \norm{\Delta_{\leq A}}_t &\leq \left( \sum_{i \leq A}  \frac{\Delta_i^2}{p_i^b} \right)^\frac{1}{2}\left( \sum_{i \leq A} p_i^r \right)^\frac{2-t}{2t}. \label{holder} 
\end{align}

Moreover, we have $\frac{1}{n^2} \leq \frac{\norm{\Delta}_t^4}{c^4\norm{p_{\leq I}}_r^2}$ so that the term $\circled{1}$ writes:

\begin{align}
    \frac{2}{n^2}\sum_{i \leq A} p_i^r & \leq 2 \sum_{i \leq A} p_i^r \left(\sum_{i\leq A} \Delta_i^t\right)^\frac{4}{t} \frac{1}{c^4 \left(\sum_{i\leq I} p_i^r\right)^\frac{2}{r}} \nonumber\\
    & \leq \frac{2}{c^4} \left(\sum_{i\leq A} p_i^r\right)^{1 - \frac{2}{r}} \left( \sum_{i \leq A}  \frac{\Delta_i^2}{p_i^b} \right)^2\left( \sum_{i \leq A} p_i^r \right)^\frac{4-2t}{t} \; \text{ by \eqref{holder}}\nonumber\\
    & = \frac{2}{c^4} \left( \sum_{i \leq A}  \frac{\Delta_i^2}{p_i^b} \right)^2 = \frac{2}{c^4} \mathbb{E}[T_{\bulk}]^2. \label{exp_bulk_H1}
\end{align}

{\bf Term $\circled{2}$:} The condition $a \leq p_A^\frac{b}{2}$ ensures that:
\begin{align*}
    p_A^b \geq a^2 &= \frac{\cA^2}{\sqrt{2}(\sum_{j \leq I}p_j^r)^{1/2} n} =: \widetilde{c} \frac{1}{(\sum_{j \leq I}p_j^r)^{1/2} n}.
\end{align*}

Using this condition, the term $\circled{2}$ writes:
\begin{equation}\label{case_2}
    \sum_{i \leq A} \frac{1}{p_{i}^{2b}} \frac{ \Delta_i^2}{n^2} \;\leq \; \frac{1}{n^2}\frac{1}{p_A^b} \sum_{i \leq A}\frac{\Delta_i^2}{p_i^b} \;\leq  \; \widetilde{c}\frac{1}{n} \left(\sum_{j \leq I}p_j^r\right)^\frac{1}{2} \left( \sum_{i \leq A}  \frac{\Delta_i^2}{p_i^b} \right).
\end{equation}

Moreover, since $ \sqrt{\frac{\norm{p_{\leq I}}_r}{n}} \lesssim \rho \lesssim \norm{\Delta_{\leq A}}_t$ we have, using (\ref{holder}):

\begin{equation}\label{r_norm_with_I}
    \frac{1}{n} \left(\sum_{j \leq I}p_j^r\right)^\frac{1}{2} = \frac{1}{n^b} \left(\sqrt{\frac{\norm{p_{\leq A}}_r}{n}}\right)^r \; \leq \; \frac{1}{n^b}\left( \sum_{i \leq A}  \frac{\Delta_i^2}{p_i^b} \right)^\frac{r}{2}\left( \sum_{i \leq A} p_i^r \right)^\frac{b}{2} \lesssim \sum_{i \leq A}  \frac{\Delta_i^2}{p_i^b}.
\end{equation}

In the last inequality, we use the fact proved in case number $\circled{1}$ that $\frac{1}{n^b} \left( \sum_{i \leq A} p_i^r \right)^\frac{b}{2} \lesssim \mathbb{E}[T_{\bulk}]^b$ and the relation $\frac{r}{2} + b = 1$\\

Plugging in (\ref{case_2}) yields that the second term $\circled{2}$ is bounded by $\mathbb{E}[T_{\bulk}]^2$\\

\hfill

{\bf Term $\circled{3}$:} This term writes:
\begin{align*}
    \frac{1}{n} \sum_{i \leq A} p_i^{1 - 2b} \Delta_i^2 &\leq \frac{\norm{\Delta_{\leq A}}_t^2}{c^2\left(\sum_{i \leq I}p_i^r \right)^\frac{1}{r}} \sum_{i \leq A} p_i^{1 - 2b} \Delta_i^2 \\
    & \leq \frac{1}{c^2} \left( \sum_{i \leq A}  \frac{\Delta_i^2}{p_i^b} \right)\left( \sum_{i \leq A} p_i^r \right)^{\frac{4-2t}{
    2t}-\frac{1}{r}}\sum_{i \leq A} p_i^{1 - 2b} \Delta_i^2 \;\;\;\;\; \text{using (\ref{holder})}\\
    & \leq \frac{1}{c^2} \left( \sum_{i \leq A}  \frac{\Delta_i^2}{p_i^b} \right)\left( \sum_{i \leq A} p_i^r \right)^\frac{-1}{2}\left(\sum_{i \leq A} p_i^{\frac{2}{3}(1 - 2b)} \Delta_i^\frac{4}{3} \right)^\frac{3}{2} \;\;\;\;\; \text{since $\norm{\cdot}_1 \leq \norm{\cdot}_\frac{2}{3}$}.
\end{align*}
Moreover, by Hölder's inequality with $\frac{1}{\frac{3}{2}} + \frac{1}{3} = 1$:
\begin{align*}
    \sum_{i \leq A} p_i^{\frac{2}{3}(1 - 2b)} \Delta_i^\frac{4}{3} & \; \;\leq \;\; \left(\sum_{i\leq A} \left(\frac{p_i^{\frac{2}{3}(1 - 2b)} \Delta_i^\frac{4}{3}}{p_i^{\frac{2}{3}\frac{t}{4-t}}} \right)^\frac{3}{2}  \right)^\frac{2}{3} \left( \sum_{i\leq A}\left(p_i^{\frac{2}{3}\frac{t}{4-t}} \right)^3 \right)^\frac{1}{3} \; \;\leq \;\; \left( \sum_{i \leq A}  \frac{\Delta_i^2}{p_i^b} \right)^\frac{2}{3} \left(\sum_{i \leq A} p_i^r \right)^\frac{1}{3}.
\end{align*}

So that 
\begin{align*}
    &\left(\sum_{i \leq A} p_i^{\frac{2}{3}(1 - 2b)} \Delta_i^\frac{4}{3} \right)^\frac{3}{2} \leq \left( \sum_{i \leq A}  \frac{\Delta_i^2}{p_i^b} \right)\left( \sum_{i \leq A} p_i^r \right)^\frac{1}{2}\\
    ie \;\;\;\;\;   & \left( \sum_{i \leq I} p_i^r \right)^\frac{-1}{2} \left(\sum_{i \leq A} p_i^{\frac{2}{3}(1 - 2b)} \Delta_i^\frac{4}{3} \right)^\frac{3}{2} \leq \left( \sum_{i \leq A}  \frac{\Delta_i^2}{p_i^b} \right).
\end{align*}

This yields that the third term satisfies:
\begin{align*}
    \frac{1}{n} \sum_{i \leq A} p_i^{1 - 2b} \Delta_i^2 \leq \frac{1}{c^2} \left( \sum_{i \leq A}  \frac{\Delta_i^2}{p_i^b} \right)^2 = \frac{1}{c^2} \mathbb{E}[T_{\bulk}]^2.
\end{align*}

{\bf Term $\circled{4}$:} The fourth term writes:

\begin{align*}
    \frac{1}{n}\norm{\left(\frac{\abs{\Delta_i}}{p_i^\frac{2b}{3}}\right)_{i \leq A}}_3^3 & \;\leq \; \frac{1}{n}\norm{\left(\frac{\abs{\Delta_i}}{p_i^\frac{2b}{3}}\right)_{i \leq A}}_2^3 \;=\; \frac{1}{n}\left(\sum_{i \leq A} \frac{\Delta_i^2}{p_i^\frac{4b}{3}} \right)^\frac{3}{2} \; \leq \; \frac{1}{n^\frac{1}{2}}\left(\sum_{i\leq A} \frac{\Delta_i^2}{p_i^b} \right)^\frac{3}{2}\left(\sum_{i\leq I} p_i^r \right)^\frac{1}{4},
\end{align*}
where in the last step we have used the fact that
$$ p_i^\frac{b}{3} \geq \frac{1}{\left(\sum_{i \leq I} p_i^r \right)^\frac{1}{6}n^\frac{1}{3} }.$$

Then using (\ref{case_2}):
\begin{align*}
 &\frac{1}{\sqrt{n}}\left(\sum_{i \leq I} p_i^r \right)^\frac{1}{4} \lesssim \left(\sum_{i\leq A} \frac{\Delta_i^2}{p_i^b} \right)^\frac{1}{2}.
\end{align*}
So the term $\circled{4}$ is upper-bounded by $\frac{1}{c^2} \mathbb{E}[T_{\bulk}]^2$. 

\paragraph{Conclusion} By Chebyshev's inequality, the type-II error of $\psi_{Bulk}$ is bounded as 
\begin{align*}
    \mathbb{P}\left(\psi_{Bulk} = 0\right) &= \mathbb{P}\left(T_{Bulk} \leq \frac{\uc}{n} \norm{p_{\leq A}}_r^\frac{r}{2}\right) = \mathbb{P}\left(\mathbb{E}(T_{Bulk}) - T_{Bulk}  \geq \mathbb{E}(T_{Bulk}) - \frac{\uc}{n} \norm{p_{\leq A}}_r^\frac{r}{2}\right)\\
    & \leq \mathbb{P}\left(\abs{\mathbb{E}(T_{Bulk}) - T_{Bulk} } \geq \mathbb{E}(T_{Bulk}) - \frac{\uc}{n} \norm{p_{\leq A}}_r^\frac{r}{2}\right)\\
    & \leq \frac{\mathbb{V}(T_{Bulk})}{\left(\mathbb{E}(T_{Bulk}) - \frac{\uc}{n} \norm{p_{\leq A}}_r^\frac{r}{2}\right)^2} \;\;\;\;\; \text{ by Chebyshev's inequality}\\
    & \leq \frac{c'\mathbb{E}(T_{Bulk})^2}{\left(\mathbb{E}(T_{Bulk}) - \frac{\uc}{n} \norm{p_{\leq A}}_r^\frac{r}{2}\right)^2}.
\end{align*}

Moreover, using \eqref{exp_bulk_H1}, we have that for $c$ large enough, $\mathbb{E}(T_{Bulk}) \geq  \frac{c}{n} \norm{p_{\leq A}}_r^\frac{r}{2} \geq 2 \frac{\uc}{n} \norm{p_{\leq A}}_r^\frac{r}{2}$ so that the denominator is well defined. Finally, since $\lim_{c \to +\infty} c' = 0$, the type-II error of this test goes to $0$ as $c$ goes to infinity,  so for $c$ large enough, the type-II error is upper-bounded by $\eta/6$

\end{proof}

We now move to the proof of Proposition \ref{prop:scase}

\hfill

\begin{proof}[Proof of Proposition~\ref{prop:scase}]

We will need the two following lemmas:

\begin{lemma}\label{A_and_1_over_n}
It holds by definition of $A$ that:
$\norm{p_{>A}}_2^2 \leq \frac{C_A}{n^2}$ for $C_A = \cA^2 + \cI$.
\end{lemma}

\begin{proof}[Proof of lemma \ref{A_and_1_over_n}]

If $A=I$ then the result is clear, by definition of $I$. Otherwise, by definition of $A$ :
    \begin{align*}
        p_{A+1}^{2b}\sum_{i \leq I}p_i^r < \frac{\cA^4}{n^2} \;\;\Longrightarrow \;\;  p_{A+1}^{2b}\sum_{i = A+1}^I p_i^r < \frac{\cA^4}{n^2} \;\;\Longrightarrow\;\;  \sum_{i = A+1}^I p_i^2 < \frac{\cA^4}{n^2} \;\;\Longrightarrow\;\;  \sum_{i >A} p_i^2 < \frac{\cA^4 + \cI}{n^2}.
    \end{align*}
\end{proof}

\begin{lemma}\label{n2_pj2}
For fixed $j>A$, the probability that coordinate $j$ is observed at least twice is upper-bounded by $n^2p_j^2$.
\end{lemma}

\begin{proof}[Proof of lemma \ref{n2_pj2}]
The probability that coordinate $j$ is observed at least twice is 
$$ 1 - (1-p_j)^n - np_j(1-p_j)^{n-1} \leq 1 - (1-np_j) - np_j[1-(n-1)p_j] \leq n^2p_j^2$$
\end{proof}

\underline{\textbf{Under $H_0$}}: We upper bound the type-I error of tests $\psi_1$ and $\psi_2$. For $\psi_2$: by Lemma \ref{A_and_1_over_n}, $p(\psi_2 = 1) \leq \sum_{j>A} n^2 p_j^2 \leq C_A \leq \frac{\eta}{4}$.

As to test $\psi_1$: $p(\psi_1 = 1) = p(|T_1| > \uc \sqrt{\frac{\sum_{i>A} p_i}{n}}) \leq \frac{\eta}{4}$ by Chebyshev's inequality. By union bound, the type-I error of $\psi_1 \vee \psi_2$ is less than $\eta/2$.\\

\underline{\textbf{Under $H_1$}}: If $\norm{\Delta_{>A}}_t \geq c\rho$, we now show that either $\psi_1$ or $\psi_2$ will detect it. \textbf{Until the end of the proof, we drop from now on the indexation ``$>A$'' and write only e.g. $\norm{p}_2, \norm{\Delta}_2$ instead of $\norm{p_{>A}}_2, \norm{\Delta_{>A}}_2$.}\\

We have by Hölder's inequality:
$$ \|\Delta\|_2^{2(t-1)} \|\Delta\|_1^{2-t} \geq \|\Delta\|_t^t \geq C \left(\frac{\|p\|_1^{2-t}}{n^{2t-2}} + \frac{1}{n^t}\right) = C \frac{1}{n^{2t-2}} \left(\|p\|_1^{2-t} + \frac{1}{n^{2-t}}\right)$$ for $C= C_1 C_2 $ where $C_1 = \left(\big(\frac{20}{\eta}(\uc + 1) +1\big)\right)^{2-t}$,\\ $C_2 = \left(\frac{1}{4}\left(\log(4/\eta)^2 \vee 9/100 \right)+ \cI\right)^{(t-1)/2}$ so that one of the two relations must hold:
$$ \|\Delta\|_2^{2(t-1)} \geq C_2 \frac{1}{n^{2t-2}} ~~ \text{ or } ~~ \|\Delta\|_1^{2-t} \geq C_1 \left(\|p\|_1^{2-t} + \frac{1}{n^{2-t}}\right)$$

\begin{itemize}
    \item First case: $\|\Delta\|_2^{2(t-1)} \geq C_2/ n^{2t-2}$. Then $\|\Delta\|_2 \geq C_2^{1/(t-1)}/n$ so that $\|q\|_2 \geq C_2^{1/(t-1)}/n - \|p\|_2 \geq \frac{1}{n}\left( C_2^{1/(t-1)} - \cI\right)$. 
    
    $\psi_2$ accepts if, and only if, all coordinates are observed at most once. This probability corresponds to:
    \begin{align*}
        q(\forall j>A, N_j = 0 \text{ or } N_j = 1) &= \prod_{j>A} \left[(1-q_j)^n + nq_j(1-q_j)^{n-1}\right]\\
        & = \prod_{j>A}(1-q_j)^{n-1}(1+(n-1)q_j)\\
        & = \prod_{j>A}(1-q_j)^{n'}(1+n'q_j), \text{ writing } n'=n-1
    \end{align*}
    Let $I_{-} = \{j>A: nq_j \leq \frac{1}{2}\}$ and $I_{+} = \{j>A: nq_j > \frac{1}{2} \}$. Recall that for $x\in (0,1/2], \; \log(1+x) \leq x - x^2/3$. Then, for $j \in I_-$: 
    \begin{align*}
        (1-q_j)^{n'}(1+n'q_j) &= \exp\left\{n' \log(1-q_j) +\log(1+n'q_j)\right\}\\
        & \leq \exp\left\{-n' q_j +n'q_j - \frac{n'^{\; 2}q_j^2}{3} \right\}\\
        & = \exp\left(- \frac{n'^{\; 2}q_j^2}{3}\right)
    \end{align*}
    
    Now, for $j \in I_+$, we have: $  n'\log(1-q_j) + \log(1+n'q_j)$ \\$\leq -n'q_j + \log(1+n'q_j) \leq -\frac{1}{10} n'q_j
$ using the inequality $-0.9x + \log(1+x) \leq 0$ true for all $x \geq \frac{1}{2}$. 
Therefore, we have upper bounded the type-II error of $\psi_2$ by:
\begin{align*}
    q(\psi=0) &\leq \exp\left(-\frac{1}{3} \sum_{j \in I_-} n'^{\; 2}q_j^2 - \frac{1}{10} \sum_{j\in I_+} n'q_j\right)\\
    & \leq \exp\left(-\frac{1}{3} \sum_{j \in I_-} n'^{\; 2}q_j^2 - \frac{1}{10}\left(\sum_{j\in I_+} n'^{\; 2}q_j^2\right)^{1/2}\right)\\
    & = \exp\left(-\frac{1}{3} (S-S_+) - \frac{1}{10}\left(S_+\right)^{1/2}\right) \text{for $S = \sum_{j>A} n'^{\; 2}q_j^2$ and $S_+ = \sum_{j\in I_+} n'^{\; 2}q_j^2$.} 
\end{align*}
Now, $S_+ \mapsto -\frac{S}{3} + \frac{1}{3}S_+ - \frac{\sqrt{S_+}}{10}$ is convex over $[0,S]$ so its maximum is reached on the boundaries of the domain and is therefore equal to $(-\frac{\sqrt{S}}{10}) \vee -\frac{S}{3} = -\frac{\sqrt{S}}{3}$ for $S \geq 9/100$. Now, since  since $\|q\|_2^2 \geq \frac{C_2^{2/(t-1)}}{n^2} \geq 4 \frac{C_2^{2/(t-1)}}{n'^{\; 2}}$, we have $S = n'^{\; 2} \|q\|_2^2 \geq \log(4/\eta)^2 \vee 9/100$ which ensures $q(\psi_2 = 0) \leq \eta/4$.

    \item Second case: $\|\Delta\|_1^{2-t} \geq C_1 \left(\|p\|_1^{2-t} + \frac{1}{n^{2-t}}\right)$. Then \\$\|\Delta\|_1 \geq C_1^{1/(2-t)} \left(\|p\|_1 \vee \frac{1}{n}\right) \geq \frac{C_1^{1/(2-t)}}{2} \left(\|p\|_1 + \frac{1}{n}\right)$. We will need the following lemma:
    \begin{lemma}\label{sum_to_norm}
    If $\sum_{j>A} \Delta_j \geq 3\sum_{j>A} p_j$ then $\left|\sum_{j>A} \Delta_j \right| \geq \frac{1}{2} \|\Delta\|_1$
    \end{lemma}
    
    \begin{proof}
    Define $J_+ = \{j>A: q_j \geq p_j\}$ and $J_- = \{q_j< p_j\}$. Define also:
    $$ s = \frac{\sum_{j>A} \Delta_j }{\sum_{j>A} p_j}, ~~~~~~ s_+ = \frac{\sum_{j \in J_+} \Delta_j }{\sum_{j>A} p_j}, ~~~~~~ s_- = - \frac{\sum_{j \in J_-} \Delta_j }{\sum_{j>A} p_j} $$
    
    Then by assumption: $s_+-s_- = s\geq 3$. Moreover, $s_- =  \frac{\sum_{j \in J_-} p_j - q_j }{\sum_{j>A} p_j}  \leq 1$. Thus, $s_+ \geq3 \geq 3s_-$ so that $2(s_+-s_-) \geq s_++s_-$, which yields the result.
    \end{proof}
    
    \hfill
    
    Note that by definition of the second case, we have for some constant $C$ that $C \|p\|_1 \leq \|\Delta\|_1 \leq \|q\|_1 + \|p\|_1$, hence that $\|q\|_1 \geq (C - 1) \|p\|_1$ and therefore taking $C\geq 5$ ensures that the assumption of Lemma \ref{sum_to_norm} are met.
    
    We can now upper bound the type-II error of $\psi_1$:
    
    \begin{align*}
        q(\psi_1 = 0) &= q\Big(\big|\sum_{j>A} \frac{N_j}{n} - p_j\big| \leq \uc \sqrt{\frac{\|p\|_1}{n}}\Big)\\
        &\leq q\Big(\big|\sum_{j>A} q_j-p_j\big| - \big|\sum_{j>A} \frac{N_j}{n} - q_j\big| \leq \uc \sqrt{\frac{\|p\|_1}{n}}\Big) \text{  by triangular inequality}\\
        & \leq q\Big(\frac{1}{2}\|q-p\|_1 -  \uc \sqrt{\frac{\|p\|_1}{n}}\leq  \big|\sum_{j>A} \frac{N_j}{n} - q_j\big| \Big) ~~ \text{ by Lemma \ref{sum_to_norm}}\\
        & \leq \frac{\frac{1}{n} \sum_{j>A} q_j}{\left(\frac{1}{2}\|q-p\|_1 -  \uc \sqrt{\frac{\|p\|_1}{n}}\right)^2} ~~ \text{ by Chebyshev's inequality}\\
        & \leq \frac{ \|q\|_1/n}{\left(\frac{1}{2}\|q\|_1 - \frac{1}{2}\|p\|_1 -  \uc \sqrt{\frac{\|p\|_1}{n}}\right)^2} \text{ by triangular inequality}\\
        & \leq \frac{ \|q\|_1/n}{\left(\frac{1}{2}\|q\|_1 - \frac{1}{2}\|p\|_1 -  \uc (\|p\|_1 + 1/n)\right)^2} \text{ using } \sqrt{xy} \leq x + y\\
        & \leq \frac{ \|q\|_1/n}{\left(\frac{1}{2}\|q\|_1 -  (\uc + 1) (\|p\|_1 + 1/n)\right)^2} \text{ using } \sqrt{xy} \leq x + y
    \end{align*}
    Now set $z = (\uc + 1) (\|p\|_1 + 1/n)$. The function $f: x \mapsto \frac{ x}{n\left(x/2 -  z\right)^2}$ is decreasing. Moreover, for $x \geq 20z/ \eta$, we have:
    \begin{align*}
        f(x) \leq \frac{20z/ \eta}{n (10z / \eta - z)^2} = \frac{20 \eta}{nz (10-\eta)^2} \underset{nz \leq 1}{\leq } \frac{20\eta}{81} \leq \eta/4
    \end{align*}
    \end{itemize}
    which proves that, whenever $\|q\|_1 \geq \frac{20}{\eta}(\uc + 1) (\|p\|_1 + 1/n)$, we have $q(\psi_1=0) \leq \eta/4$. This condition is guaranteed when $\|\Delta\|_1 \geq \big(\frac{20}{\eta}(\uc + 1) +1\big)(\|p\|_1 + 1/n) = C_1^{1/(2-t)}(\|p\|_1 + 1/n)$

\end{proof}

\hfill

\begin{proof}[Proof of lemma \ref{expVar}]


\begin{itemize}
    \item \underline{Expectation:} 
    \begin{align*}
        \mathbb{E}\left[T_{\bulk}\right] & =\sum_{i \leq A}\;\; \frac{1}{p_i^{2b}} \left(\mathbb{E}\left[\frac{S_i}{\bar n} - p_i\right]\mathbb{E}\left[\frac{S'_i}{\bar n} - p_i\right]\right)\\
        & = \sum_{i \leq A} \;\; \frac{1}{p_i^{2b}} \left(p_i - q_i \right)^2.
    \end{align*}
    \item \underline{Variance:} 
        \begin{align*}
            \mathbb{V}(T_{\bulk}) &= \sum_{i \leq A}\;\;\; \frac{1}{p_i^{2b}} \left(\mathbb{E}\left[\left(\frac{S_i}{\bar n} - p_i\right)^2 \left(\frac{S'_i}{\bar n} - p_i\right)^2  \right] - \mathbb{E}\left[\left(\frac{S_i}{\bar n} - p_i\right)\left(\frac{S'_i}{\bar n} - p_i\right)  \right]^2\right)\\
            & = \sum_{i \leq A}\; \;\; \frac{1}{p_i^{2b}} \left(\mathbb{E}\left[ \left(\frac{S_i}{\bar n} - p_i \right)^2 \right]^2 - \left(p_i - q_i \right)^4\right),
            \end{align*}
            Since the $(S_i, S_i')_i$ are independent. And so by a bias-variance decomposition, and since $S_i, S'_i \sim \mathcal B(\bar n, q_i)$
             \begin{align*}
            \mathbb{V}(T_{\bulk})& = \sum_{i \leq A} \;\; \frac{1}{p_i^{2b}} \left( \left[ \mathbb{V}\left(\frac{S_i}{\bar n} \right) + \mathbb{E}\left[ \left(\frac{S_i}{\bar n} - p_i \right) \right]^2\right]^2 - \left(p_i - q_i \right)^4\right)\\
            & = \sum_{i \leq A}\;\; \frac{1}{p_i^{2b}} \left(\left[\frac{q_i(1-q_i)}{\bar n} + \left(p_i - q_i \right)^2 \right]^2 - \left(p_i - q_i \right)^4\right)\\
            & = \sum_{i \leq A}\;\; \frac{1}{p_i^{2b}} \left(\frac{q_i^2(1-q_i)^2}{\bar n^2} + \frac{2}{\bar n}q_i(1-q_i)\left(p_i - q_i \right)^2\right)\\
            & \leq \sum_{i \leq A}\;\; \frac{1}{p_i^{2b}} \left(\frac{q_i^2}{\bar n^2} + \frac{2}{\bar n}q_i\left(p_i - q_i \right)^2\right).
        \end{align*}
\end{itemize}
\end{proof}

\begin{proof}[Proof of lemma \ref{expVarT1}]

We therefore have
\begin{align*}
    \mathbb{E}[T_1] &= \mathbb{E}\left[\sum_{i >A} \frac{S_i + S'_i}{n} - p_i\right] = \sum_{i >A} p_i - q_i,
    \end{align*}
    and 
    \begin{align*}
    \mathbb{V}[T_1] &= \mathbb{V}\left[\sum_{i >A} \frac{S_i + S'_i}{n}\right] = \sum_{i >A} \frac{\mathbb{V}\left[S_i\right] + \mathbb{V}\left[S'_i\right]}{n^2}~~~~~\text{ by independence of the }~(S_i, S'_i)_i\\
    &\\
    & = \sum_{i>A} \frac{q_i(1-q_i)}{n} \leq \sum_{i>A} \frac{q_i}{n}
\end{align*}
\end{proof}

\begin{proof}[Proof of lemma \ref{expVarT2}] The proof is similar to the proof of lemma \ref{expVar}, by replacing $b$ with $0$ and summing over $i>A$ instead of $i \leq A$.
\end{proof}

\hfill\break

\section{Equivalence between the Binomial, Poisson and Multinomial settings}

We now prove that the rates for goodness of fit testing in the Binomial, Poisson and Multinomial case are equivalent.  

\hfill

\begin{proof}[Proof of Lemma \ref{bin_eq_poi}]
\textbf{We first prove} $\rho_{Poi}(n,p) \leq C_{BP}\;\rho_{Bin}(n,p)$. Let $n \geq 2$, and let $Y_1, \cdots, Y_n \overset{iid}{\sim} Poi(q)$. We consider a random function $\phi$ such that for any Poisson family $Y_1, \cdots, Y_n \overset{iid}{\sim} Poi(q)$, 
$$\left\{\begin{array}{ll}
     \phi(Y_1, \cdots, Y_n) = (X_1, \cdots, X_{\widetilde n}) \overset{iid}{\sim} Ber(q)  \;\;\; \text{ where } \;\;\; \widetilde n \sim Poi(n) \indep (Y_i)_i \\
     \\
     \sum_{i=1}^{\widetilde n} X_i = \sum_{i=1}^{n} Y_i
\end{array} \right.$$ 

\noindent In words, $\phi$ is a function which takes $n$ Poisson random variables (or equivalently one Poisson random variable $Poi(nq)$) and decomposes them into $\widetilde n \sim Poi(n)$ Bernoulli iid random variables whose sum is  $\sum_{i=1}^n Y_i$.\\ 

\noindent Let $\widetilde n \sim Poi(n)$ be the random length of $\phi(Y_1, \cdots, Y_n)$. We can choose a small constant $c = c(\eta)$ such that the event:
$$ \mathcal{A}_1:= \left\{\widetilde n \; \geq \; cn \right\}$$

\noindent has probability larger than $1 - \eta/4$. Moreover, for $m \geq cn$ we can define the function
$$ \pi(x_1, \cdots, x_m) = (x_1, \cdots, x_{\lfloor cn \rfloor})$$

\noindent Let $\psi_{Bin}$ be the test associated to the \textbf{binomial} testing problem:
$$ H_0: q = p \;\;\;\;\; \text{ v.s. } \;\;\;\;\; H_1: \norm{p-q}_t \geq \rho_{Bin}(cn, p, \frac{\eta}{2})$$
In particular, $R(\psi_{Bin}) \leq \eta/2$. Now, we define the test
$$ \psi = \left\{\begin{array}{ll}
     \psi_{Bin} \circ \pi \circ \phi \;\;\; \text{ if } \;\;\; \mathcal{A}_1\\
     \\
     0 \; \text{ otherwise} 
\end{array} \right.$$
\noindent and we show that, when associated to the \textbf{Poissonian} testing problem
$$ H_0: q = p \;\;\;\;\; \text{ v.s. } \;\;\;\;\; H_1: \norm{p-q}_t \geq \rho$$
\noindent with $\rho = \rho_{Bin}(cn, p, \frac{\eta}{2})$, it has a risk less than $\eta$. We first analyse its type-I error.
\begin{align*}
    \mathbb{P}_{H_0}\left(\psi(Y_1^n) =1\right) & \leq  \mathbb{P}_{H_0}\left(\mathcal{A}_1 \cap \psi(Y_1^n)=1 \right) + \mathbb{P}_{H_0}(\bar{\mathcal{A}_1})\\
    & \leq \mathbb{P}_{H_0}\left(\psi(Y_1^n)=1 | \mathcal{A}_1\right) + \frac{\eta}{4}\\
    & \leq \mathbb{P}_{H_0}\left(\psi_{Bin}(X_1, \cdots, X_{\lfloor cn\rfloor})=1 | \mathcal{A}_1\right) + \frac{\eta}{4}\\
    & = \mathbb{P}_{X_1^{\lfloor cn\rfloor} \sim Ber( p)^{\bigotimes \lfloor cn\rfloor}}\left(\psi_{Bin}(X_1, \cdots, X_{\lfloor cn\rfloor})=1\right) + \frac{\eta}{4}
\end{align*}

For the Type-II error, the same steps show that for any vector $q$:
\begin{align*}
    \mathbb{P}_q\left(\psi(Y_1^n) =0\right) \leq \mathbb{P}_{X_1^{\lfloor cn\rfloor} \sim Ber( q)^{\bigotimes \lfloor cn\rfloor}}\left(\psi_{Bin}(X_1, \cdots, X_{\lfloor cn\rfloor})=0\right) + \frac{\eta}{4}
\end{align*}
We can now compute the risk of $\psi$ when $\rho = \rho_{Bin}(cn, p, \frac{\eta}{2})$:
\begin{align*}
    R(\psi) & = \mathbb{P}_{H_0}\left(\psi(Y_1^n) =1\right) + \sup_{\norm{p-q}_t \geq \rho} \mathbb{P}_q\left(\psi(Y_1^n) =0\right)\\
    & \leq \frac{\eta}{2} + \mathbb{P}_{X_1^{\lfloor cn\rfloor} \sim Ber( p)^{\bigotimes \lfloor cn\rfloor}}\left(\psi_{Bin}(X_1, \cdots, X_{\lfloor cn\rfloor})=1\right)\\
    & \;\;\;\;\;\;\;\;+ \sup_{\norm{p-q}_t \geq \rho} \mathbb{P}_{X_1^{\lfloor cn\rfloor} \sim Ber(q)^{\bigotimes \lfloor cn\rfloor}}\left(\psi_{Bin}(X_1, \cdots, X_{\lfloor cn\rfloor})=0\right) \\
    &\\
    & = \frac{\eta}{2} + R(\psi_{Bin})\\
    & = \frac{\eta}{2}+\frac{\eta}{2} = \eta
\end{align*}
This proves $\rho_{Poi}(n,p) \leq \rho_{Bin}(cn, p, \frac{\eta}{2}) \asymp \rho_{Bin}(n, p, \eta)$.\\

\textbf{We now show }$\rho_{Poi}(n,p) \geq c_{BP}\;\rho_{Bin}(n,p)$. Let $X_1, \cdots, X_n \sim Ber(q)$ iid. For some small constant $\overline c >0$ let $\widetilde n \sim Poi(\lfloor \overline c n \rfloor)$. We choose $\overline c>0$ such that 
\begin{equation}\label{A2}
    \mathcal{A}_2 = \left\{ \widetilde n \leq  n \right\}
\end{equation} has probability larger than $1-\frac{\eta}{4}$. Consider the extended sequence of multivariate Bernoulli random variables $(\widetilde X_i)_i$ such that
$$ \left\{ \begin{array}{ll}
     \widetilde X_i = X_i &\text{ if } i \leq n  \\
     \widetilde X_i \sim Ber(q) &\text{ otherwise }
\end{array} \right. $$
and such that $(\widetilde X_i)_i$ are mutually independent. Let $Y = \sum_{i=1}^{\widetilde n} X_i \sim Poi(\lfloor \overline c n\rfloor q)$. The sum is a sufficient statistic of the parameter $q$ for Poisson random variables so we can define a function
$$ \bar \phi (Y) = (Y_1, \cdots, Y_{\lfloor \overline c n\rfloor})$$
such that $Y_i \sim Poi(q)$ iid. Moreover, we set for $m \leq n$:
$$ \Bar{\pi}(y_1,\cdots, y_n, m) = (y_1, \cdots, y_m)$$
On $\mathcal{A}_2$, we do not even need to extend the sequence of observations. We call $\psi_{Poi}$ the test associated to the \textbf{Poisson} testing problem:
$$ H_0: q = p \;\;\;\;\; \text{ v.s. } \;\;\;\;\; H_1: \norm{p-q}_t \geq \rho_{Poi}(\lfloor \overline c n \rfloor, p, \frac{\eta}{2})$$
We define the randomized test
\begin{equation}\label{test_Poi_for_Bin}
    \bar \psi = \left\{\begin{array}{ll}
     \psi_{Poi} \circ \bar \pi \circ \bar \phi(Y) & \text{if } \mathcal{A}_2 \\
     &\\
     0 & \text{otherwise.}
\end{array} \right.
\end{equation}
We show that this test has a risk less than $\eta$. For the type-I error:
\begin{align*}
    \mathbb{P}_{H_0}\left(\bar \psi(Y) =1\right) & \leq  \mathbb{P}_{H_0}\left(\mathcal{A}_2 \cap \bar \psi(Y) =1\ \right) + \mathbb{P}_{H_0}(\bar{\mathcal{A}_2})\\
    & \leq \mathbb{P}_{H_0}\left(\bar \psi(Y) =1\ | \mathcal{A}_2\right) + \frac{\eta}{4}\\
    & \leq \mathbb{P}_{H_0}\left(\psi_{Poi}(Y_1, \cdots, Y_{\lfloor \overline c n\rfloor})=1 | \mathcal{A}_2\right) + \frac{\eta}{4}\\
    & = \mathbb{P}_{Y_1^{\lfloor \overline c n\rfloor} \sim Poi(p)^{\bigotimes \lfloor \overline c n\rfloor}}\left(\psi_{Poi}(Y_1, \cdots, Y_{\lfloor \overline c n\rfloor})=1\right) + \frac{\eta}{4}
\end{align*}

For the Type-II error, the same steps show that for any vector $q$:
\begin{align*}
    \mathbb{P}_q\left(\bar \psi(Y) =0\right) \leq \mathbb{P}_{Y_1^{\lfloor \overline c n\rfloor} \sim Poi(q)^{\bigotimes \lfloor \overline c n\rfloor}}\left(\psi_{Poi}(Y_1, \cdots, Y_{\lfloor \overline c n\rfloor})=0\right) + \frac{\eta}{4}
\end{align*}

We can now compute the risk of $\bar \psi$ when $\rho = \rho_{Poi}(\overline c n, p, \frac{\eta}{2})$:
\begin{align*}
    R(\psi) & = \mathbb{P}_{H_0}\left(\bar\psi(Y) =1\right) + \sup_{\norm{p-q}_t \geq \rho} \mathbb{P}_q\left(\psi(Y) =0\right)\\
    & \leq \frac{\eta}{2} + \mathbb{P}_{Y_1^{\lfloor \overline c n\rfloor} \sim Poi(p)^{\bigotimes \lfloor \overline c n\rfloor}}\left(\psi_{Poi}(Y_1, \cdots, Y_{\lfloor \overline c n\rfloor})=1\right)\\
    & \;\;\;\;\;\;\;\;+ \sup_{\norm{p-q}_t \geq \rho} \mathbb{P}_{Y_1^{\lfloor \overline c n\rfloor} \sim Poi(q)^{\bigotimes \lfloor \overline c n\rfloor}}\left(\psi_{Poi}(Y_1, \cdots, Y_{\lfloor \overline c n\rfloor}) = 0 \right) \\
    &\\
    & = \frac{\eta}{2} + R(\psi_{Poi})\\
    & = \frac{\eta}{2}+\frac{\eta}{2} = \eta
\end{align*}
This proves $\rho_{Bin}(n,p) \leq \rho_{Poi}(\overline c n, p, \frac{\eta}{2}) \asymp \rho_{Poi}(n, p, \eta)$.\\

\end{proof}

\hfill

\begin{proof}[Proof of Lemma \ref{mult_eq_poi}]
We first prove that $\rho_{Mult}(n,p) \lesssim \rho_{Poi}(n, p^{-\max})$ when $\sum p_i = 1$ by following the same steps as for proving $\rho_{Bin} \lesssim \rho_{Poi}$: we draw $\widetilde n \sim Poi(\overline c n)$ and $Z_1, \cdots, Z_{\widetilde n} \overset{iid}{\sim} \mathcal{M}(q)$. Then the histogram (or fingerprints) is a sufficient statistic of $Z_1, \cdots, Z_{\widetilde n}$ for $q$. It is defined as 
$$ \begin{pmatrix}N_1\\ \vdots\\ N_d\end{pmatrix} := \begin{pmatrix}\sum_{i=1}^{\widetilde n} \mathbb{1}\{Z_i = 1\}\\ \vdots\\ \sum_{i=1}^{\widetilde n} \mathbb{1}\{Z_i = d\}\end{pmatrix} \sim Poi(nq).$$\\
On $\mathcal{A}_2$, defined in \eqref{A2}, we have 
$$ \begin{pmatrix}N_2\\ \vdots\\ N_d\end{pmatrix} \sim Poi(n (q_2, \cdots, q_d))$$ so we can just apply the exact same steps to prove that, if $q = p$ then the test $\bar \psi$ from \eqref{test_Poi_for_Bin} has type-I error less than $\frac{\eta}{2}$ and if $\|q - p\|_{\mathcal{M},t} \geq \rho_{Poi}(\overline{c} n, p, \frac{\eta}{2})$, its type-II error is less than $\frac{\eta}{2}$.\\

We now prove the converse lower bound: $\rho_{Poi}^*(n,p,\eta) \leq \rho_{Mult}^*(n,p,\eta)$. For this, we come back to the prior distributions defined in \eqref{Prior_bulk} and \eqref{prior_tail} except that we \textit{do not set} any perturbation on $p_1$. This defines a probability distribution $\widetilde p$ such that $\forall\;  1 < j \leq A, \widetilde p_j = p_j + \delta_j \gamma_j$ where $\delta_j \sim Rad(\frac{1}{2})$ iid and $\forall j \geq U, \widetilde p_j = \bar \pi b_j$ where $b_j \sim Ber(\frac{p_j}{\bar \pi})$ iid and $\bar \pi = \frac{c_u}{n^2 \|P_{\geq U}\|_1}$. We will project $\widetilde p$ onto the simplex so that it is a probability vector. Define $p' = \frac{\widetilde p}{\|\widetilde p\|_1}$. $p'$ therefore follows a prior distribution on the set of $d$-dimensional probability vectors. We now show that this prior concentrates on a zone located at $\rho_{Poi}^*(n,p)$ from $p$ (up to a constant), and that it is undetectable when observing $n$ iid data drawn from $p'$ where $p'$ follows this prior.\\

Consider the high probability event 

$$ \mathcal{A}_3 = \big\{\big|\big(\sum_{j=2}^A p_j + \delta_j \gamma_j\big) + \bar \pi \|p_{\geq U}\|_1  - \sum_{j=2}^d p_j\big| \leq \frac{c_1}{\sqrt{n}} + \frac{c'_1}{n} \big\}$$ for a small constant $c_1$, and 

$$ \mathcal{A}_4 = \big\{\big|\sum_{j=2}^A \delta_j \big| \leq c_2\sqrt{A} \big\}$$

The following lemma proves that $\mathcal{A}_3$ is a high probability event:

\begin{lemma}\label{A3_high_prob}
There exist two constants $c_1,c'_1$ such that $\mathbb{P}(\mathcal{A}_3)\geq 1 - \frac{\eta}{4}$
\end{lemma}

\textbf{For the bulk}, define 
$$ J_+ = \{2 \leq j \leq A: \delta_j = 1\} \text{ and } J_- = \{2 \leq j \leq A: \delta_j = -1\}$$
On $\mathcal{A}_4$, we have $\|\gamma_{J_+}\|_t \gtrsim \|\gamma\|_t$ and $\|\gamma_{J_-}\|_t \gtrsim \|\gamma\|_t$. There are two cases:
\begin{itemize}
    \item First case: $\| \widetilde p \|_1 \geq 1$. Then $p'$ is obtained by shrinking $\widetilde p$. This means that \\ $\big\|\big( \frac{\widetilde p}{\|\widetilde p\|_1} \big)_{J_-} - p_{J_-}\big\|_t \geq \big\|\widetilde p_{J_-}  - p_{J_-}\big\|_t = \|\gamma_{J_-}\|_t \gtrsim \|\gamma\|_t$, where we define $p_{J_-} = (p_j)_{j \in J_-}$ and all other quantities similarly.
    
    \item Second case:  $\| \widetilde p \|_1 < 1$. Then similarly: $\big\|\big( \frac{\widetilde p}{\|\widetilde p\|_1} \big)_{J_+} - p_{J_+}\big\|_t \gtrsim \|\gamma\|_t $.
\end{itemize}
In both cases, the rescaled vector $p'$ is still separated away from the null distribution by a distance at least $\rho_{Poi, Bulk}^*(n,p)$.

\textbf{For the tail}: On $\mathcal{A}_3$, we have $\|\widetilde p\|_1 \in [\frac{1}{2}, \frac{3}{2}]$ so that $\forall \; j\geq U: \frac{\widetilde p_j}{\|\widetilde p\|_1}  \asymp \widetilde p_j$. The exact same calculation as in the proof of lemma \ref{tail} shows that, with high probability, $\|p'_{\geq U} - p_{\geq U}\|_t^t \gtrsim \frac{\|p_{\geq U}\|_1^{2-t}}{n^{2(t-1)}}$. Combining the above results, we get that $\|p' - p\|_t \geq \rho_{Poi}^*(n,p)$ and that this prior is indistinguishable from the null distribution, ensuring $\rho_{Poi}^*(n,p) \lesssim \rho_{Mult}^*(n,p)$. Indeed, on $\mathcal{A}_3$, the rescaling factor is between $1/2$ and $3/2$ so that on the bulk, we still have $\forall j = 2, \dots, A: |p'_j - p| \leq 2\gamma_j$ and the perturbation $(\pm \gamma_j)_j$ is already undetectable.
\end{proof}

\begin{proof}[Proof of Lemma \ref{A3_high_prob}]
We prove the lemma in two steps: first, we prove that the bulk prior satisfies \textit{whp}:
\begin{equation}\label{prior_bulk_close}
\big|\sum_{j\leq A} p_j - \widetilde p_j\big| \leq \frac{c_b}{\sqrt{n}}
\end{equation}
and second, we prove that the sparse prior on the tail satisfies \textit{whp}
\begin{equation}\label{prior_tail_close}
    \big|\sum_{j\geq U} p_j - \widetilde p_j\big| \leq \frac{c_s}{\sqrt{n}} + \frac{c'_s}{n}
\end{equation}
For both inequalities we use Chebyshev's inequality by computing the expectations and variances of the both priors and by proving that the standard deviation is smaller than the expectation.

\begin{itemize}
    \item \textbf{\underline{Bulk}}: We have 
    \begin{align*}
        \big|\sum_{j\leq A} p_j - \widetilde p_j\big| = \big| \sum_{j\leq A} \gamma_j\delta_j \big|
    \end{align*}
    We have: $\mathbb{E}\left[\sum_{j\leq A} \gamma_j\delta_j\right] = 0 $ and $\mathbb{V}\left[\sum_{j\leq A} \gamma_j\delta_j\right] = \sum_{j\leq A} \gamma_j^2$. Moreover,
    \begin{align*}
        \sum_{j\leq A} \gamma_j^2 &= \sum_{j\leq A} \; \frac{\cA^2 \; p_j^{4/(4-t)}}{n \, \big(\sum_{j\leq A} p_j^r\big)^{1/2}} \leq p_A^b \sum_{j\leq A} p_j^{4/(4-t)}\\
        & \leq \sum_{j\leq A} p_j^2 \leq \big(\sum_{j\leq A} p_j\big)^2
    \end{align*}
    
    \hfill
    
    \item \textbf{\underline{Tail}}: We have 
    \begin{align*}
        \mathbb{E}\left[\sum_{j\geq U} p_j - \widetilde p_j \right] & = \mathbb{E}\left[\sum_{j\geq U} b_j \bar \pi - p_j \right] = 0\\
        \mathbb{V}\left[\sum_{j\geq U} p_j - \widetilde p_j \right] & = \bar{\pi}^2 \sum_{j\geq U} \frac{p_j}{\bar{\pi}} = \frac{c_u}{n^2} \leq \frac{c_u}{n} + \|p_{\geq U}\|_1^2 \asymp \big(\frac{1}{n} + \|p_{>A}\|_1\big)^2
    \end{align*}
\end{itemize}
Therefore, by Chebyshev's inequality, \textit{whp} the prior we set concentrates on a zone such that $$\big|\big(\sum_{j=2}^A p_j + \delta_j \gamma_j\big) + \bar \pi \|p_{\geq U}\|_1  - \sum_{j=2}^d p_j\big| \leq \frac{c_1}{\sqrt n} + \frac{c'_1}{n}$$
\end{proof}

\section{Tightness of~\cite{balakrishnan2019hypothesis} in the multinomial case}\label{sec:Wasserman_tight}

\revision{
For fixed $n$ and for two absolute constants $C,c>0$, define $\epsilon_+$ as the largest quantity satisfying $\epsilon_+ \leq C\sqrt{\frac{\|p^{-\max}_{-\epsilon_+/16}\|_{2/3}}{n}}+ \frac{C}{n}$ and $\epsilon_-$ as the smallest quantity satisfying $\epsilon_- \geq c\sqrt{\frac{\|p^{-\max}_{-\epsilon_-}\|_{2/3}}{n}} + \frac{c}{n}$. By~\cite{balakrishnan2019hypothesis}, the critical radius $\rho^*$ satisfies $\epsilon_- \lesssim \rho^* \lesssim \epsilon_+$.
\begin{enumerate}
    \item First case: If $\epsilon_+ \leq 16 \epsilon_-$, then the bounds match.
    \item Second case: otherwise, $\epsilon_+ \leq C\sqrt{\frac{\|p^{-\max}_{-\epsilon_+/16}\|_{2/3}}{n}}+ \frac{C}{n} \leq  C\sqrt{\frac{\|p^{-\max}_{-\epsilon_-}\|_{2/3}}{n}}+ \frac{C}{n} \leq \frac{C}{c} \epsilon_-$ so that the bounds also match in this case.
\end{enumerate}
}

\bibliographystyle{plain}

\end{document}